\documentclass[10pt]{article}
\usepackage{amsfonts}
\usepackage{amsfonts,latexsym,amsmath,amscd,geometry}
\usepackage{amssymb}
\usepackage{latexsym}
\usepackage{cases}

\usepackage{amsmath,amssymb,amsthm,amsfonts}
\usepackage{mathrsfs}
\usepackage{bbm}
\usepackage{color}

\usepackage{mathrsfs}
\usepackage{graphicx,psfrag,graphics}

\newtheorem{lemma}{Lemma}[section]
\newtheorem{theorem}{Theorem}[section]

\newtheorem{proposition}{Proposition}[section]

\numberwithin{equation}{section} \arraycolsep=1.5pt

\newcommand{\R}{\mathbb{R}}

\newcommand{\al}{\alpha}

\newcommand{\pa}{\partial}

\newcommand{\eps}{\epsilon}

\begin{document}

\title{{\bf {Stability of the superposition of boundary layer and rarefaction wave
 for outflow problem on the two-fluid Navier-Stokes-Poisson system }}}

\author{
\begin{tabular}{cc}
& H{\sc aiyan}  Y{\sc in$^{1}$},\ \ \ J{\sc inshun} Z{\sc
hang}$^{2}$\ \ \ {\&} \
C{\sc hangjiang} Z{\sc hu$^{3}$}\\
&
  {\small\it 1.
      School of Mathematical Sciences,}\\
&
  {\small\it
      Huaqiao
University, Quanzhou 362021, P.R. China}\\
&
   {\small\it 2.
      School of
Mathematical Sciences,} \\
& {\small\it
         Huaqiao University, Quanzhou 362021, P.R.
China}\\
 &
  {\small\it 3.
       School of Mathematics,}\\
&  {\small\it
       South China
University of Technology, Guangzhou 510641, P.R. China}\\
&
 {\small\it
        E-mail: yinhaiyan2000@aliyun.com,  jszhang@hqu.edu.cn,   cjzhu@mail.ccnu.edu.cn}\\
\end{tabular}
}

\date{}

\maketitle

\begin{abstract}
This paper is concerned with the study of nonlinear stability of
  superposition of boundary layer and rarefaction wave on the two-fluid
 Navier-Stokes-Poisson system in the half line
$\R_{+}=:(0,+\infty)$. On account of the quasineutral assumption and
the absence of the electric field for the large time behavior, we
successfully construct the boundary layer and rarefaction wave, and
then we give the rigorous proofs of the stability theorems on the
superposition of boundary layer and rarefaction wave under small
perturbations for the corresponding initial boundary value problem
of the Navier-Stokes-Poisson system, only provided the strength of
boundary layer is small while the strength of  rarefaction wave can
be arbitrarily large. The complexity of nonlinear composite wave
leads to many complicated terms in the course of establishing the
{\it a priori} estimates. The proofs are given by an elementary
$L^2$ energy method.
\end{abstract}

\medskip


\medskip


\tableofcontents \vspace{4mm}

\section{Introduction}
\subsection{The problem}
The dynamics of the charged particles in the collisional dusty
plasma can be described by the Navier-Stokes-Poisson (called NSP in
the sequel for simplicity) system which reads in the Eulerian
coordinates
\begin{eqnarray}\label{1.1}
&&\left\{\begin{aligned}
& \pa_{t}\rho_{i}+\pa_{x}(\rho_{i} u_{i})=0,\\
&\rho_{i}(\pa_{t}u_{i}+u_{i}\pa_{x}u_{i})+\pa_{x}P(\rho_{i})
=\rho_{i}E+\mu_{i} \pa_{x}^{2}u_{i},\\
& \pa_{t}\rho_{e}+\pa_{x}(\rho_{e} u_{e})=0,\\
&\rho_{e}(\pa_{t}u_{e}+u_{e}\pa_{x}u_{e})+\pa_{x}P(\rho_{e})
=-\rho_{e}E+\mu_{e} \pa_{x}^{2}u_{e},\\
&\pa_{x}E=\rho_{i}-\rho_{e}.
\end{aligned}\right.
\end{eqnarray}
 Here, for $\alpha=i,e$, $P(\rho_{\alpha})$ is
pressure which is given by
\begin{eqnarray}\label{1.3}
P(\rho_{\alpha})=A\rho_{\alpha}^{\gamma_{\alpha}},
\end{eqnarray}
where $A$ is a positive constant and $\gamma_{\alpha}> 1$ is the
adiabatic exponent. Thus each fluid (ions or electrons) is regarded
as an ideal polytropic gas. The unknown functions $\rho_{\alpha}$
and $u_{\alpha}$
 stand for the density and velocity of ions $(\alpha=i)$ and  electrons $(\alpha=e)$ in
plasma, respectively, and $E$ is the electric field, while the
positive constants $\mu_{\alpha}>0$ denote the viscosity coefficient
 of ions $(\alpha=i)$ and electrons
$(\alpha=e)$, respectively. Throughout the paper, for brevity we
assume $\gamma_{i}=\gamma_{e}=\gamma> 1$;
  the case of $\gamma_{i}\neq\gamma_{e}$ and $\gamma_{i}=\gamma_{e}=1$
could be considered in a similar way. We also assume
$\mu_{i}=\mu_{e}=1$ throughout the paper. One can see \cite{{F.
Chen}} and \cite{{PM-1990}} for more information about the physical
background of model $\eqref{1.1}$.

 We consider \eqref{1.1} in the half line $\mathbb{R}_{+}$ with initial data
\begin{eqnarray}\label{1.a}
[\rho_{i},u_{i},\rho_{e},u_{e}](x,0)=[\rho_{i0},u_{i0},\rho_{e0},u_{e0}](x)
\rightarrow [\rho_{+},u_{+},\rho_{+},u_{+}]\ \ \ \mathrm{as} \ \  x
\rightarrow +\infty,
\end{eqnarray}
where $\rho_{+}>0$ and $u_{+}$ are constants. The boundary
conditions are
\begin{eqnarray}\label{1.b}
u_{i}(0,t)=u_{e}(0,t)=u_{b}<0,  \quad \forall\, t\geq 0,
\end{eqnarray}
 and
the compatibility condition $u_{b}=u_{i0}(0)=u_{e0}(0)$ holds.

In the case of $u_{b}<0$, electrons and ions fluids flow away from
the boundary $\{x=0\}$, and thus the problem \eqref{1.1},
\eqref{1.a} and \eqref{1.b} in such case is called an outflow
problem. The case of $u_{b}=0$ and $u_{b}>0$ is called the
impermeable wall problem and the inflow problem, respectively.
Notice that for the inflow problem, there should been an additional
boundary condition on the density. In the paper, we focus on the
outflow problem in the case of $u_{b}<0.$ Here we remark that the
impermeable wall problem and the inflow problem of the
 Navier-Stokes-Poisson system are left for study
in the future.

\subsection{Some preliminary}
In order to study the large time behavior of solutions to the
initial boundary value problem \eqref{1.1}, \eqref{1.a} and
\eqref{1.b}, we notice that in the simplified case of the electric
field $E=0$  and the quasineutral assumptions $\rho_{i}=\rho_{e}$
and $u_{i}=u_{e}$ for the large time behavior, the problem is
reduced to consider the following single quasineutral Navier-Stokes
equation
\begin{eqnarray}\label{1.3}
&&\left\{\begin{aligned}
& \pa_{t}\rho+\pa_{x}(\rho u)=0,\\
&\rho(\pa_{t}u+u\pa_{x}u)+\pa_{x}P(\rho) =\pa_{x}^{2}u
\end{aligned}\right.
\end{eqnarray}
with initial data
\begin{equation}
\label{dd-id} [\rho,u](x,0)=[\rho_{0},u_{0}](x)\rightarrow
[\rho_{+},u_{+}],\quad as \quad x\rightarrow+\infty
\end{equation}
and the boundary condition
\begin{equation}
\label{ddw1-bd} u(0,t)=u_{b}<0, \quad \forall\, t\geq 0.
\end{equation}

 Matsumura \cite{MNinflow} gave the classification of the large
time behavior solutions to the outflow problem for Navier-Stokes
equation \eqref{1.3} in terms of $(\rho_{+},u_{+})$ and $u_{b}<0.$
In what follows, let us recall some basic facts concerning the study
of the outflow problem. The characteristic speeds of the hyperbolic
part of \eqref{1.3} are
\begin{eqnarray}\label{1.5}
\lambda_{1}=u-C(\rho),\ \ \ \lambda_{2}=u+C(\rho),
\end{eqnarray}
where $C(\rho)=\sqrt{P'(\rho)}=\sqrt{\gamma
A}\rho^{\frac{\gamma-1}{2}}$ is the local sound speed. From now on,
we define
$$v=\frac{1}{\rho},\ \  v_{+}=\frac{1}{\rho_{+}},\cdots, and \  so \  on,$$
 where $v$ is the specific volume.
 Let
$$C_{+}=C(\rho_{+})=\sqrt{\gamma A}\rho_{+}^{\frac{\gamma-1}{2}}=\sqrt{\gamma A}v_{+}^{-\frac{\gamma-1}{2}},\ \ \
M_{+}=\frac{|u_{+}|}{C_{+}}$$ be the sound speed and the Mach number
at the far field $x=+\infty$, respectively. The phase plane
$\mathbb{R}_{+}\times \mathbb{R}$ of $(v,u)$ can be divided into
three subsets:
$$\Omega_{sub}:=\left\{(v,u)\in \mathbb{R}_{+}\times
\mathbb{R};\ \ \ |u|<C\left(\frac{1}{v}\right)\right\},$$
$$\Gamma_{trans}:=\left\{(v,u)\in \mathbb{R}_{+}\times
\mathbb{R};\ \ \ |u|=C\left(\frac{1}{v}\right)\right\},$$
$$\Omega_{super}:=\left\{(v,u)\in \mathbb{R}_{+}\times
\mathbb{R}; \ \ \ |u|>C\left(\frac{1}{v}\right)\right\},$$ where
$\Omega_{sub},$ $\Gamma_{trans}$ and $\Omega_{super}$ are called the
subsonic, transonic and supersonic regions, respectively. In the
phase plane, we denote the curves through a right state point
$(v_{1},u_{1})$:
$$BL(v_{1},u_{1})=\left\{(v,u)\in \mathbb{R}_{+}\times
\mathbb{R}; \ \ \ \frac{u}{v}=\frac{u_{1}}{v_{1}}\right\},$$
$$R_{2}(v_{1},u_{1})=\left\{(v,u)\in \mathbb{R}_{+}\times
\mathbb{R};\ \ \ u=u_{1}-\sqrt{\gamma
A}\int_{v_{1}}^{v}s^{-\frac{\gamma+1}{2}}ds,\ \ \ v>v_{1}\right\},$$
$$S_{2}(v_{1},u_{1})=\left\{(v,u)\in \mathbb{R}_{+}\times
\mathbb{R};\ \ \
u=u_{1}+\sqrt{\left[P\left(\frac{1}{v}\right)-P\left(\frac{1}{v_{1}}\right)\right](v_{1}-v)},\
\ \ v<v_{1}\right\},$$ to be the boundary line, 2-rarefaction wave
and 2-shock wave curves, respectively. Then the large time behavior
of solutions to the outflow problem \eqref{1.3}, \eqref{dd-id} and
\eqref{ddw1-bd} can be classified into the following four cases (the
cases are omitted which concern  shock waves):
\begin{figure}[h!]
\begin{center}
\psfrag{1}{$u_{b}$}
  \psfrag{2}{$u_{b}$}
  \psfrag{3}{$O$}
  \psfrag{4}{$(v_{*},u_{*})$}
  \psfrag{5}{$(v_{+},u_{+})$}
  \psfrag{6}{$BL(v_{+},u_{+})$}
  \psfrag{7}{$R_{2}(v_{+},u_{+})$}
  \psfrag{8}{$S_{2}(v_{+},u_{+})$}
  \psfrag{9}{$v$}
  \psfrag{0}{$u$}
\includegraphics[width=0.3\textwidth]{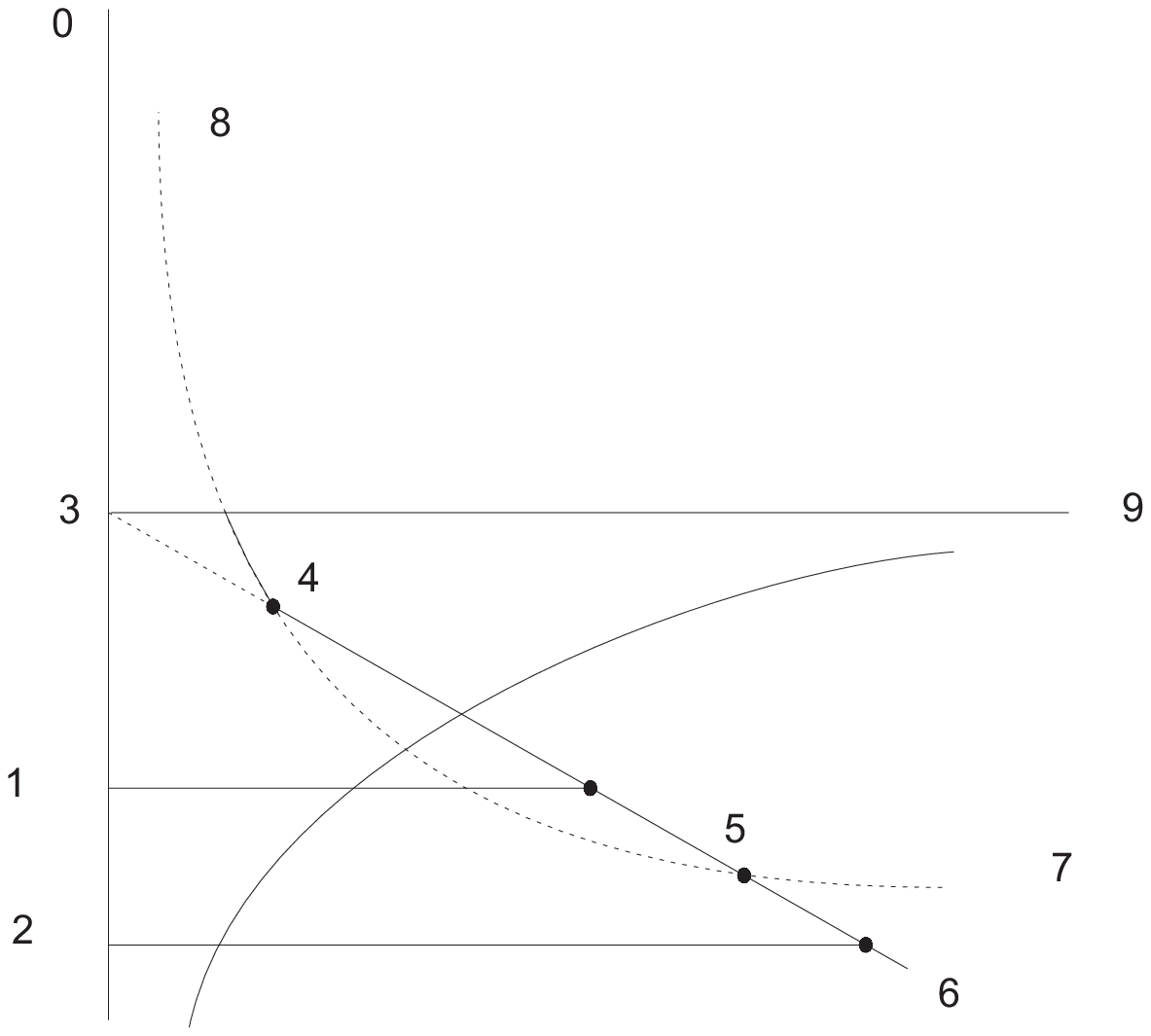}
\end{center}
\begin{center}
Figure 1
\end{center}
\end{figure}

\medskip

\textbf{Case I:} $(v_{+},u_{+})\in \Omega_{super}\bigcap\{u_{+}<0\}$
and $u_{b}<u_{*}$. Here $(v_{*},u_{*})$ is an intersection point of
$BL(v_{+},u_{+})$ and $S_{2}(v_{+},u_{+})$, ie.,
\begin{eqnarray}\label{1.6}
u_{+}=\frac{u_{+}}{v_{+}}v_{*}-
\sqrt{\left[P\left(\frac{1}{v_{*}}\right)-P\left(\frac{1}{v_{+}}\right)\right](v_{+}-v_{*})},\
\ \ u_{*}=\frac{u_{+}}{v_{+}}v_{*}.
\end{eqnarray}
Then there exists a unique $v_{b}$ such that $(v_{b},u_{b})\in
BL(v_{+},u_{+})$, and the time asymptotic state of solution is a
boundary layer $(\tilde{v},\tilde{u})(x)$ which connects
$(v_{b},u_{b})$ with $(v_{+},u_{+})$, see Figure 1. By the relation
of $\rho$ and $v$, then we can say that boundary layer
$(\tilde{\rho},\tilde{u})(x)$ connects $(\rho_{b},u_{b})$ with
$(\rho_{+},u_{+}).$ The boundary layer $(\tilde{\rho},\tilde{u})(x)$
will be explained in next section.

\medskip

 \textbf{Case II:}
$(v_{+},u_{+})\in \Gamma_{trans}\bigcap\{u_{+}<0\}$ and
$u_{b}<u_{+}$. Then there exists a unique $v_{b}$ such that
$(v_{b},u_{b})\in BL(v_{+},u_{+})$, and the time-asymptotic state of
solution is a boundary layer $(\tilde{v},\tilde{u})(x)$ which
connects $(v_{b},u_{b})$ with $(v_{+},u_{+})$, see Figure 2. Here,
the boundary layer $(\tilde{v},\tilde{u})(x)$ is degenerate. That is
to say  boundary layer $(\tilde{\rho},\tilde{u})(x)$ connects
$(\rho_{b},u_{b})$ with $(\rho_{+},u_{+})$, and the boundary layer
$(\tilde{\rho},\tilde{u})(x)$ is degenerate.

\begin{figure}[h!]
\begin{center}
\psfrag{1}{$u_{b}$}
  \psfrag{3}{$O$}
  \psfrag{5}{$(v_{+},u_{+})$}
  \psfrag{6}{$BL(v_{+},u_{+})$}
  \psfrag{7}{$R_{2}(v_{+},u_{+})$}
  \psfrag{8}{$S_{2}(v_{+},u_{+})$}
  \psfrag{9}{$v$}
  \psfrag{0}{$u$}
\includegraphics[width=0.3\textwidth]{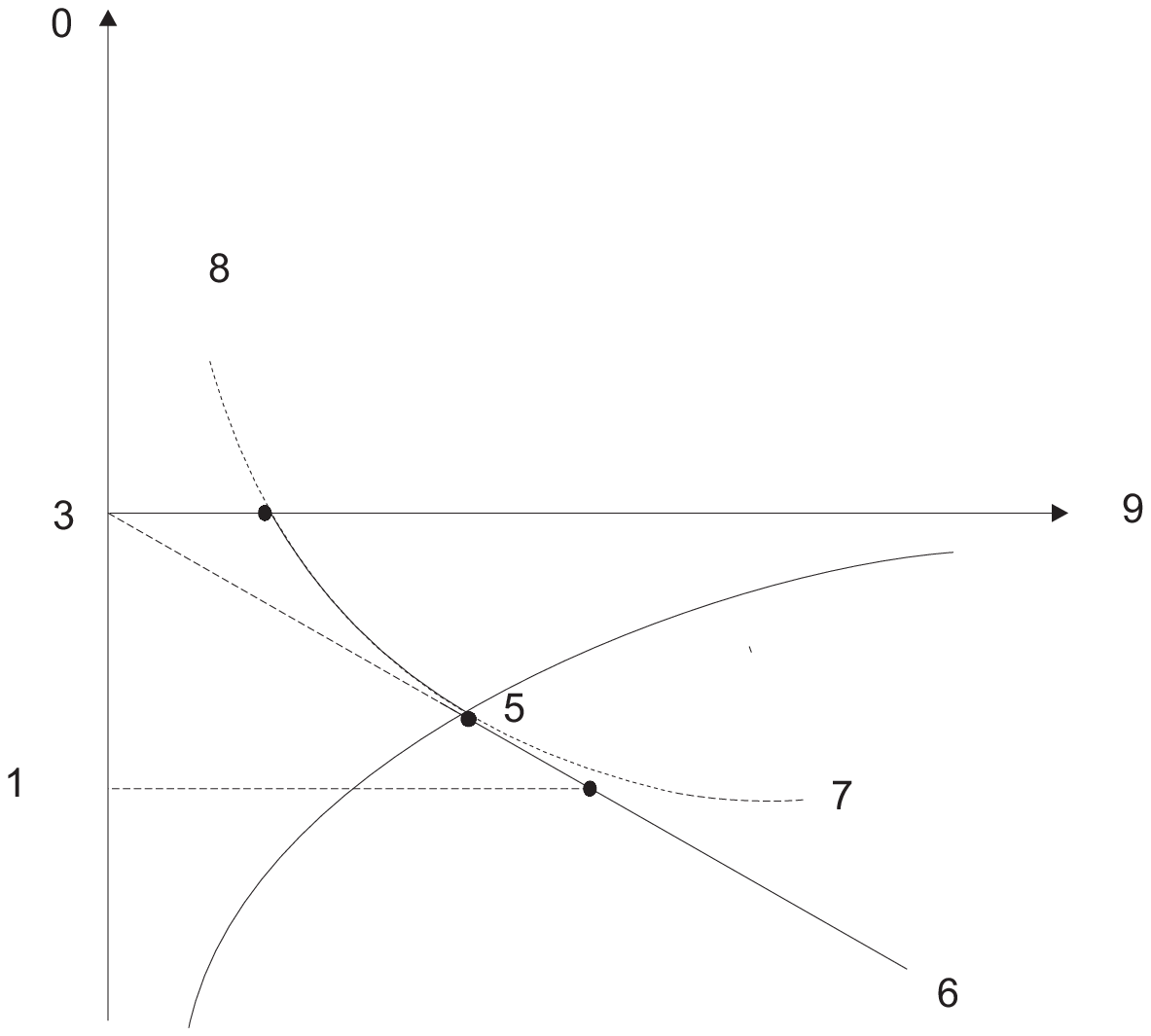}
\end{center}
\begin{center}
Figure 2
\end{center}
\end{figure}

\medskip

\textbf{Case III:} $(v_{+},u_{+})\in \Omega_{sub}\bigcap\{u_{+}<0\}$
and $u_{b}<u_{+}$. Here $(v_{*},u_{*})$ is an intersection point of
$R_{2}(v_{+},u_{+})$ and $\Gamma_{trans}$, ie.,
\begin{eqnarray}\label{1.6}
u_{+}-\sqrt{\gamma
A}\int_{v_{+}}^{v_{*}}s^{-\frac{\gamma+1}{2}}ds=-\sqrt{\gamma
A}v_{*}^{-\frac{\gamma-1}{2}},\ \ \ u_{*}=-\sqrt{\gamma
A}v_{*}^{-\frac{\gamma-1}{2}},
\end{eqnarray}
see Figure 3. This case is divided into two subcases:

\textbf{Subcase 1:} If $u_{*}\leq u_{b}<u_{+},$ then there exists a
unique $v_{b}$ such that $(v_{b},u_{b})\in R_{2}(v_{+},u_{+})$, and
the time-asymptotic state of solution is a 2-rarefaction wave
$(v^{R_{2}},u^{R_{2}})(\frac{x}{t}),$ which connects $(v_{b},u_{b})$
with $(v_{+},u_{+})$, to the corresponding Riemann problem,  while
the 2-rarefaction wave $(\rho^{R_{2}},u^{R_{2}})(\frac{x}{t})$
connects $(\rho_{b},u_{b})$ with $(\rho_{+},u_{+})$.

\textbf{Subcase 2:} If $u_{*}> u_{b},$ then there exists a unique
$v_{b}$ such that $(v_{b},u_{b})\in BL(v_{*},u_{*})$, and the
time-asymptotic state of solution is the superposition of a boundary
layer $(\tilde{v},\tilde{u})(x)$ connecting $(v_{b},u_{b})$ with
$(v_{*},u_{*})$, which is degenerate, and a 2-rarefaction wave
$(v^{R_{2}},u^{R_{2}})(\frac{x}{t})$  connecting $(v_{*},u_{*})$
with $(v_{+},u_{+})$, while boundary layer
$(\tilde{\rho},\tilde{u})(x)$ connects $(\rho_{b},u_{b})$ with
$(\rho_{*},u_{*})$, and a 2-rarefaction wave
$(\rho^{R_{2}},u^{R_{2}})(\frac{x}{t})$ connects $(\rho_{*},u_{*})$
with $(\rho_{+},u_{+})$.

\begin{figure}[h!]
\begin{center}
\psfrag{1}{$u_{b}$}
  \psfrag{2}{$u_{b}$}
  \psfrag{3}{$O$}
  \psfrag{4}{$(v_{*},u_{*})$}
  \psfrag{5}{$(v_{+},u_{+})$}
  \psfrag{6}{$BL(v_{*},u_{*})$}
  \psfrag{7}{$R_{2}(v_{+},u_{+})$}
  \psfrag{8}{$S_{2}(v_{+},u_{+})$}
  \psfrag{9}{$v$}
  \psfrag{0}{$u$}
\includegraphics[width=0.3\textwidth]{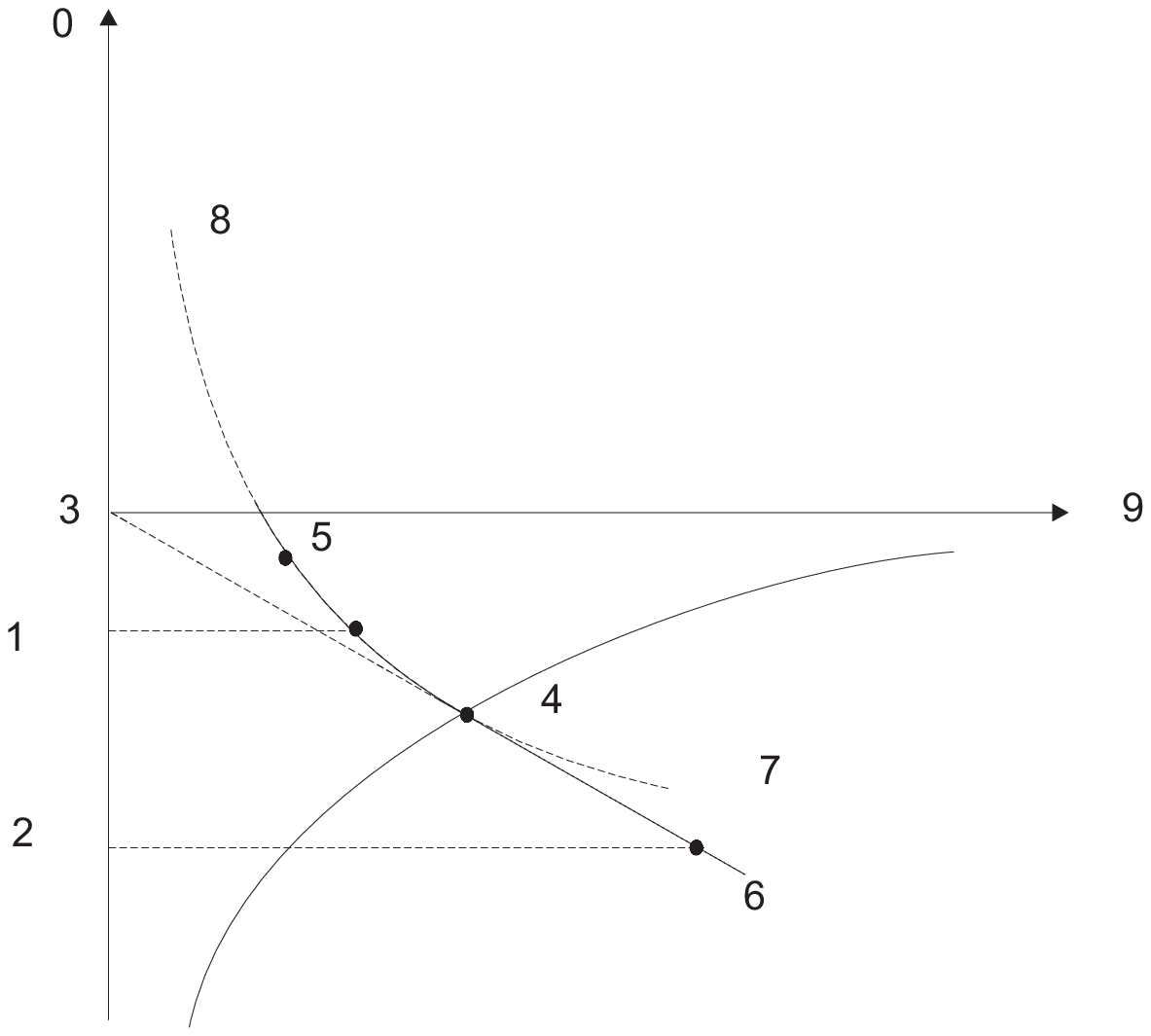}\\
\end{center}
\begin{center}
Figure 3
\end{center}
\end{figure}

\medskip

\textbf{Case IV:} $u_{+}>0$ and $u_{b}<0$. Here $(v_{*},u_{*})$ is
an intersection point of $R_{2}(v_{+},u_{+})$ and $\Gamma_{trans}$
which is defined by \eqref{1.6}, see Figure 4. This case is divided
into two subcases:

\textbf{Subcase 1:} If $u_{*}\leq u_{b}<0,$ then there exists a
unique $v_{b}$ such that $(v_{b},u_{b})\in R_{2}(v_{+},u_{+})$, and
the time-asymptotic state of solution is a 2-rarefaction wave
$(v^{R_{2}},u^{R_{2}})(\frac{x}{t}),$ which connects $(v_{b},u_{b})$
with $(v_{+},u_{+})$, to the corresponding Riemann problem, while a
2-rarefaction wave $(\rho^{R_{2}},u^{R_{2}})(\frac{x}{t})$ connects
$(\rho_{b},u_{b})$ with $(\rho_{+},u_{+}).$

\textbf{Subcase 2:} If $u_{*}> u_{b},$ then there exists a unique
$v_{b}$ such that $(v_{b},u_{b})\in BL(v_{*},u_{*})$, and the
time-asymptotic state of solution is the superposition of a boundary
layer $(\tilde{v},\tilde{u})(x)$ connecting $(v_{b},u_{b})$ with
$(v_{*},u_{*})$, which is degenerate, and a 2-rarefaction wave
$(v^{R_{2}},u^{R_{2}})(\frac{x}{t})$  connecting $(v_{*},u_{*})$
with $(v_{+},u_{+})$, while boundary layer
$(\tilde{\rho},\tilde{u})(x)$ connects $(\rho_{b},u_{b})$ with
$(\rho_{*},u_{*})$, and a 2-rarefaction wave
$(\rho^{R_{2}},u^{R_{2}})(\frac{x}{t})$ connects $(\rho_{*},u_{*})$
with $(\rho_{+},u_{+})$.

\begin{figure}[h!]
\begin{center}
 \psfrag{1}{$u_{b}$}
  \psfrag{2}{$u_{b}$}
  \psfrag{3}{$O$}
  \psfrag{4}{$(v_{*},u_{*})$}
  \psfrag{5}{$(v_{+},u_{+})$}
  \psfrag{6}{$BL(v_{*},u_{*})$}
  \psfrag{7}{$R_{2}(v_{+},u_{+})$}
  \psfrag{8}{$S_{2}(v_{+},u_{+})$}
  \psfrag{9}{$v$}
  \psfrag{0}{$u$}
\includegraphics[width=0.3\textwidth]{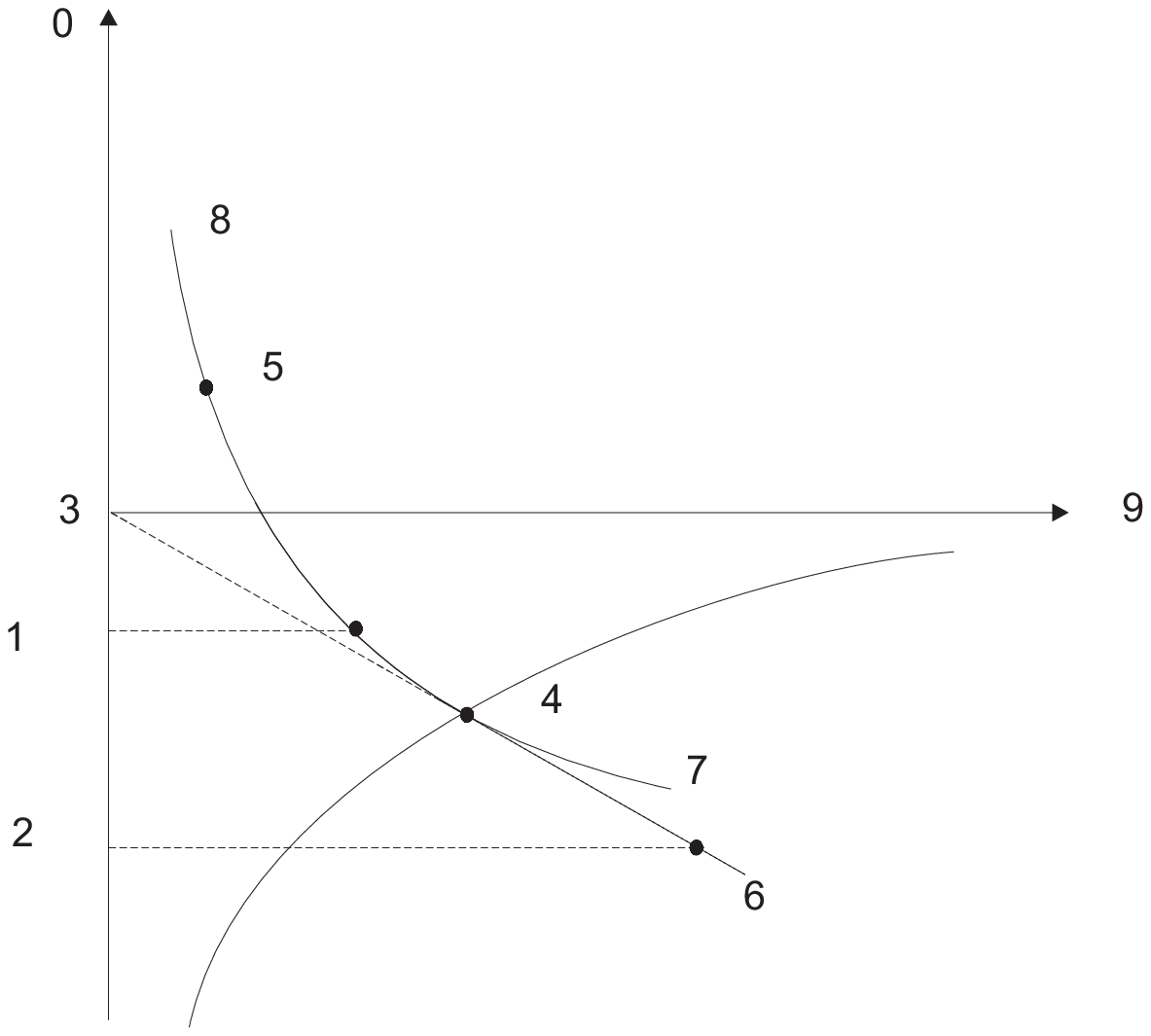}\\
\end{center}
\begin{center}
Figure 4
\end{center}
\end{figure}

\subsection{Boundary layer and  rarefaction wave}
In the paper, we study the subcase 2 in Case III or Case IV without
considering the other cases since the cases of the single wave have
been studied by Duan and Yang \cite{DY2013}. Recalling subcase 2 in
Case III or Case IV,  there exists a unique $v_{b}$ in phase plane
such that $(v_{b},u_{b})\in BL(v_{*},u_{*}),$ where $(v_{*},u_{*})$
is defined in \eqref{1.6}. And the solution to the initial boundary
value problem \eqref{1.1}, \eqref{1.a} and \eqref{1.b} for the
outflow problem on two-fluid Navier-Stokes-Poisson system  is
expected to tend to the superposition of a degenerate boundary layer
$(\tilde{\rho},\tilde{u})(x)$ connecting $(\rho_{b},u_{b})$ with
$(\rho_{*},u_{*})$ and a 2-rarefaction wave
$(\rho^{R_{2}},u^{R_{2}})(\frac{x}{t})$ connecting
$(\rho_{*},u_{*})$ with $(\rho_{+},u_{+})$ as $t$ $\rightarrow
+\infty$
 coupling the trivial profile of electric field $E=0.$

 First of all, we define the boundary layer $(\tilde{\rho},
\tilde{u})$ by the stationary solution to
\begin{eqnarray}\label{2.1}
&&\left\{\begin{aligned}
& \partial_{x}(\tilde{\rho}\tilde{u})=0,\ \ \  x\in \mathbb{R}_{+},\\
& \tilde{\rho}\tilde{u}\partial_{x}\tilde{u}+\partial_{x}P(\tilde{\rho})=\partial_{x}^{2}\tilde{u},\ \ \  x\in \mathbb{R}_{+},\\
&\tilde{u}(0)=u_{b},\ \ \ (\tilde{\rho},
\tilde{u})(+\infty)=(\rho_{*},u_{*}), \ \ \  \inf_{x\in
\mathbb{R}_{+}}\tilde{\rho}(x)>0.
\end{aligned}\right.
\end{eqnarray}
Integrating $\eqref{2.1}_{1}$ over $[x,+\infty)$ for $x>0$, and
letting $x\rightarrow 0,$ we obtain the value of $\tilde{\rho}(x)$
at the boundary $\{x=0\}$ as follows:
\begin{eqnarray}\label{2.2}
\rho_{b}:=\tilde{\rho}(0)=\frac{\rho_{*}u_{*}}{u_{b}}.
\end{eqnarray}
Since $u_{b}<0$, we have $u_{*}<0.$ The strength of the boundary
layer $(\tilde{\rho}, \tilde{u})(x)$ is measured by
\begin{eqnarray}\label{2.3}
\tilde{\delta}:=|u_{*}-u_{b}|.
\end{eqnarray}
In what follows let us present the existence and some known
properties of the boundary layer $(\tilde{\rho}, \tilde{u})(x)$
connecting $(\rho_{b},u_{b})$ with $(\rho_{*},u_{*})$ for the
stationary problem $\eqref{2.1}.$ Here we only list the properties
of the degenerate boundary layer. Please refer to \cite{KNZhd} or
\cite{MNinflow} for details.
\begin{lemma}\label{lem.Vw}
By the definition of $(v_{*},u_{*})$ in Subcase 2 in Case III or
Case IV (i.e. it is located at the transonic curve), then there
exists a solution $(\tilde{\rho}, \tilde{u})(x)$ to the stationary
problem \eqref{2.1} such that
$\tilde{u}=\frac{u_{*}}{v_{*}}\tilde{v},$
$\tilde{v}=\frac{1}{\tilde{\rho}}.$ Moreover, $\tilde{u}(x)$ is
monotonically increasing $(\partial_{x}\tilde{u}\geq 0)$ and
converges to $u_{*}$ algebraically as x tends to infinity.
Precisely, there exists a positive constant C such that
\begin{eqnarray}\label{2.4a}
|\partial_{x}^{k}[\tilde{\rho}-\rho_{*},\tilde{u}-u_{*}]|\leq
\frac{C\tilde{\delta}^{k+1}}{(1+\tilde{\delta}x)^{k+1}},\ \ \
k=0,1,2,\cdots.
\end{eqnarray}

\end{lemma}
\vspace{3mm}

Since the 2-rarefaction wave
$\left[\rho^{R_{2}},u^{R_{2}}\right](\frac{x}{t})$ is a weak
solution, we shall construct a smooth approximation for the
2-rarefaction wave above in the following. Firstly, consider the
Riemann problem for Burger's equation:
\begin{eqnarray}\label{2.5}
&&\left\{\begin{aligned} & \partial_{t}w+w\partial_{x}w=0,\\
&w(0,x)=w_{0}(x)=\left\{\begin{aligned} &w_{-},\ \ \ x<0,\\
&w_{+},\ \ \ x>0,\end{aligned}\right.
\end{aligned}\right.
\end{eqnarray}
where $w_{-}<w_{+}.$  Then it is well known that \eqref{2.5} has a
continuous weak solution $w^{R_{2}}(\frac{x}{t})$ whose explicit
form is given by
\begin{eqnarray}\label{2.6}
w^{R_{2}}(\frac{x}{t})=\left\{\begin{aligned} &w_{-},\ \ \ x<w_{-}t,\\
&\frac{x}{t},\ \ \ w_{-}t\leq x\leq w_{+}t,\\
&w_{+},\ \ \ x>w_{+}t.
\end{aligned}\right.
\end{eqnarray}
Moreover, $w^{R_{2}}(\frac{x}{t})$ can be approximated by the smooth
function $\overline{w}(t,x)$ which is a solution to
\begin{eqnarray}\label{2.df9}
&&\left\{\begin{aligned} & \pa_t\bar{w}+\bar{w}\pa_x\bar{w}=0,\\
&\bar{w}(0,x)=\bar{w}_{0}(x)=\left\{\begin{aligned} &w_{-},\ \ \ \ \ \ \  \ \ \ \ \ \  \ \ \ \ \ \ \ \ \ \ \ \ \  \ \ x<0,\\
&w_{-}+C_{q}\bar{\delta}\int_{0}^{\epsilon x}y^{q}e^{-y}dy,\ \ \
x>0,\end{aligned}\right.
\end{aligned}\right.
\end{eqnarray}
where $\bar{\delta}:=w_{+}-w_{-},\ q\geq 10$ is a constant, $C_{q}$
is a constant such that $C_{q}\int_{0}^{\infty}y^{q}e^{-y}dy=1,$ and
$\epsilon\leq 1$ is a positive constant to be determined later. Then
we have the following lemma.

\begin{lemma}\label{cl.Re.Re.}
 Let $\bar{\delta}=w_+-w_-$ be the wave strength
of the 2-rarefaction wave. Then the problem \eqref{2.df9} has a
unique smooth solution $\bar{w}(x,t)$ which satisfies the following
properties:

\noindent$(i)$ $0<w_-<\bar{w}(x,t)<w_+$, $\pa_x\bar{w}\geq0$ for
$x\in \R$ and $t\geq0$.

\noindent$(ii)$ For any $p$ $(1\leq p\leq +\infty)$, there exists a
constant $C_{p,q}$ such that for $t\geq 0$
$$\|\partial_{x}\bar{w}\|_{L^{p}}\leq C_{p,q}\min\{\bar{\delta}\epsilon^{1-\frac{1}{p}},\bar{\delta}^{\frac{1}{p}}t^{-1+\frac{1}{p}}\},$$
$$\|\partial_{x}^{2}\bar{w}\|_{L^{p}}\leq C_{p,q}\min\{\bar{\delta}\epsilon^{2-\frac{1}{p}},\bar{\delta}^{\frac{1}{q}}\epsilon ^{1-\frac{1}{p}+\frac{1}{q}}
t^{-1+\frac{1}{q}}\}.$$

 \noindent$(iii)$ When $x\leq w_{-}t,$\ \
$\bar{w}-w_{-}=\partial_{x}\bar{w}=\partial_{x}^{2}\bar{w}=0.$

 \noindent$(iv)$
$\lim\limits_{t\rightarrow+\infty}\sup\limits_{x\in\R}|\bar{w}(x,t)-w^{R_{2}}(\frac{x}{t})|=0$.

\end{lemma}

Then the smooth approximate rarefaction wave
$\left[\rho^{r_{2}},u^{r_{2}}\right](x,t)$ which corresponds to the
rarefaction wave $\left[\rho^{R_{2}},u^{R_{2}}\right](\frac{x}{t})$
can be defined as follows:
\begin{eqnarray}\label{1.18}
\begin{split}
\left\{\begin{array}{rll}
&u^{r_{2}}+C(\rho^{r_{2}})=\bar{w}(x,1+t),\  w_{-}=u_{*}+C(\rho_{*})=0,\ \ w_{+}=u_{+}+C(\rho_{+})>0,\\[2mm]
&u^{r_{2}}=u_{+}-\sqrt{\gamma
A}\int_{v_{+}}^{v^{r_{2}}}s^{-\frac{\gamma+1}{2}} ds,\ \
v^{r_{2}}=\frac{1}{\rho^{r_{2}}},\ \  v_{+}=\frac{1}{\rho_{+}},
\end{array}
\right.
\end{split}
\end{eqnarray}
where $\bar{w}(x,t)$  is given in \eqref{2.df9}.

 It is easy to
obtain $[\rho^{r_{2}},u^{r_{2}}](x,t)$
 satisfies
\begin{eqnarray}\label{ME2}
\left\{
\begin{array}{clll}
\begin{split}
& \partial_{t}\rho^{r_{2}}+\partial_{x}(\rho^{r_{2}} u^{r_{2}})=0,\\[2mm]
&\rho^{r_{2}}\pa_{t}u^{r_{2}}+\rho^{r_{2}}u^{r_{2}}\pa_{x}u^{r_{2}}+\pa_{x}P(\rho^{r_{2}})=0.
\end{split}
\end{array}
\right.
\end{eqnarray}
Here we restrict $[\rho^{r_{2}},u^{r_{2}}](x,t)$ in the half space
$\{x\geq 0\}$. Then one has
\begin{lemma}\label{cl.Re.Re2.}
Let $\delta_{r}=|\rho_+-\rho_*|+|u_+-u_*|$ be the wave strength of
the 2-rarefaction wave. Then the smooth approximate 2-rarefaction
wave $[\rho^{r_{2}},u^{r_{2}}](x,t)$  constructed in \eqref{1.18}
has the following properties:

\noindent$(i)$ $\pa_xu^{r_{2}}\geq0$,
$\rho_{*}<\rho^{r_{2}}(x,t)<\rho_{+}$,
    $u_{*}<u^{r_{2}}(x,t)<u_{+}$, $\pa_xu^{r_{2}}\thicksim |\pa_x\rho^{r_{2}}|$
for $x\in\R_{+}$ and $t\geq0$.

\noindent$(ii)$ For any $p$ $(1\leq p\leq +\infty)$, there exists a
constant $C_{p,q}$ such that for $t>0$,
$$
\|\pa_x[\rho^{r_{2}}, u^{r_{2}}]\|_{L^p(\R_{+})}\leq
C_{p,q}\min\{\delta_{r}\eps^{1-\frac{1}{p}},
\delta_{r}^{\frac{1}{p}}(1+t)^{-1+\frac{1}{p}}\},
$$
$$
\|\pa^2_x[\rho^{r_{2}}, u^{r_{2}}]\|_{L^p(\R_{+})}\leq
C_{p,q}\min\{\delta_{r}\eps^{2-\frac{1}{p}},
\delta_{r}^{\frac{1}{q}}\epsilon
^{1-\frac{1}{p}+\frac{1}{q}}(1+t)^{-1+\frac{1}{q}}\}.
$$

\noindent$(iii)$ $[\rho^{r_{2}},u^{r_{2}}](0,t)=[\rho_{*},u_{*}].$

\noindent$(iv)$
$\lim\limits_{t\rightarrow+\infty}\sup\limits_{x\in\R_{+}}\left|[\rho^{r_{2}},
u^{r_{2}}](x,t)-\left[\rho^{R_{2}},u^{R_{2}}\right](\frac{x}{t})
\right|=0$.

\end{lemma}

Now, we define
\begin{eqnarray}\label{2.11}
[\hat{\rho},\hat{u}](x,t):=[\tilde{\rho},\tilde{u}](x)+[\rho^{r_{2}},
u^{r_{2}}](x,t)-[\rho_{*},u_{*}].
\end{eqnarray}
By a straightforward calculation, we have
\begin{eqnarray}\label{2.12}
&&\left\{\begin{aligned}
&\partial_{t}\hat{\rho}+\partial_{x}(\hat{\rho} \hat{u})=\hat{f},\ \ \ (x,t)\in \mathbb{R}_{+}\times\mathbb{R}_{+},\\
&\hat{\rho}(\partial_{t}\hat{u}+
\hat{u}\partial_{x}\hat{u})+\partial_{x}P(\hat{\rho})=\partial_{x}^{2}\hat{u}+\hat{g},\
\ \ (x,t)\in \mathbb{R}_{+}\times\mathbb{R}_{+},\\
&(\hat{\rho},\hat{u})(x,0)\rightarrow (\rho_{+},u_{+}), \ \
\mathrm{as} \ \ x\rightarrow+\infty,\ \
(\hat{\rho},\hat{u})(0,t)=(\rho_{b},u_{b}).
\end{aligned}\right.
\end{eqnarray}
where
\begin{eqnarray}\label{2.13}
&&\left\{\begin{aligned}
\hat{f}=&\partial_{x}\tilde{\rho}(u^{r_{2}}-u_{*})+\partial_{x}\tilde{u}(\rho^{r_{2}}-\rho_{*})+\partial_{x}\rho^{r_{2}}(\tilde{u}-u_{*})+\partial_{x}u^{r_{2}}(\tilde{\rho}-\rho_{*}),\\
\hat{g}=&-\partial_{x}^{2}u^{r_{2}}+\tilde{u}\partial_{x}\tilde{u}(\rho^{r_{2}}-\rho_{*})
+\hat{\rho}\left[\partial_{x}\tilde{u}(u^{r_{2}}-u_{*})+\partial_{x}u^{r_{2}}(\tilde{u}-u_{*})\right]\\
&+
\partial_{x}\tilde{\rho}\left[P'(\hat{\rho})-P'(\tilde{\rho})\right]
+
\partial_{x}\rho^{r_{2}}\left[P'(\hat{\rho})-P'(\rho^{r_{2}})
\right]-
\frac{P'(\rho^{r_{2}})}{\rho^{r_{2}}}\partial_{x}\rho^{r_{2}}(\tilde{\rho}-\rho_{*}).
\end{aligned}\right.
\end{eqnarray}
From $\eqref{2.1}_{1}$ and $\eqref{2.11}$, it is easy to know
\begin{eqnarray}\label{2.14}
&&\left\{
\begin{aligned}
&|\hat{f}|+|\hat{g}+\partial_{x}^{2}u^{r_{2}}|\leq
C\left\{\partial_{x}\tilde{u}(u^{r_{2}}-u_{*})+\partial_{x}u^{r_{2}}(u_{*}-\tilde{u})\right\},\\
&|\partial_{x}\hat{f}|\leq
C\left\{(|\partial_{x}^{2}\tilde{u}|+(\partial_{x}\tilde{u})^{2})(u^{r_{2}}-u_{*})+\partial_{x}\tilde{u}\partial_{x}u^{r_{2}}+|\partial_{x}^{2}u^{r_{2}}|+(\partial_{x}u^{r_{2}})^{2}\right\},
\end{aligned}
\right.
\end{eqnarray}
where $\partial_{x}\tilde{u}\geq 0,$ $\partial_{x}u^{r_{2}}\geq 0$
and $\tilde{u}\leq u_{*}\leq u^{r_{2}}.$

\subsection{Main results}

We can easily derive
$E(x,t)=-\int_{x}^{+\infty}[\rho_{i}(y,t)-\rho_{e}(y,t)]dy$ from
$\eqref{1.1}_{5}$ if we assume that $E(x,t)\rightarrow 0\
\mathrm{as}  \ x \rightarrow +\infty$ holds. Then we can define
$E(x,0)=-\int_{x}^{+\infty}[\rho_{i0}(y)-\rho_{e0}(y)]dy.$ Now we
are in a position to state our main results.
\begin{theorem}\label{main.Res.}
Let $\alpha=i,e$ and assume that constant states $u_{b}$, $u_{*}$
and the infinite state $(\rho_{+},u_{+})$ satisfy Subcase 2 either
in Case III or in Case IV. There exist some positive constants
$\varepsilon_{0}>0$ and $C_{0}>0$ such that if
\begin{equation}\label{retghy}
\begin{split}
\left\|[\rho_{\alpha0}(\cdot)-\hat{\rho}(\cdot,0),u_{\alpha0}(\cdot)-\hat{u}(\cdot,0)]\right\|_{H^1}^{2}
+\|E(\cdot,0)\|^{2}+\epsilon^{\frac{1}{10}}+\tilde{\delta}^{\frac{1}{9}}\leq
\varepsilon_{0}^{2},
\end{split}
\end{equation}
where $\epsilon>0$ is the parameter appearing in \eqref{2.df9}, then
the initial boundary value problem \eqref{1.1}, \eqref{1.a} and
\eqref{1.b} admits a unique global solution
$[\rho_{\alpha},u_{\alpha},E](x,t)$  satisfying
\begin{equation}\label{resultqa21hl}
\begin{split}
\sup_{t\geq
0}\left\|[\rho_{\alpha}-\hat{\rho},u_{\alpha}-\hat{u},E](\cdot,t)\right\|_{H^1}\leq
C_{0}\varepsilon_{0}.
\end{split}
\end{equation}
Moreover, the solution $[\rho_{\alpha},u_{\alpha},E](x,t)$ tends
time-asymptotically to the composite wave in the sense that
\begin{equation}\label{resultqa216y89}
\begin{split}
\lim_{t\rightarrow+\infty}\sup_{x\in
\R_{+}}\left|[\rho_{\alpha},u_{\alpha}](x,t)-\left[\tilde{\rho},\tilde{u}\right](x)-\left[\rho^{R_{2}},u^{R_{2}}\right]\left(\frac{x}{t}\right)+\left[\rho_{*},u_{*}\right]\right|=0,
\end{split}
\end{equation}
and
\begin{equation}\label{resultqa21gt}
\begin{split}
\lim_{t\rightarrow+\infty}\sup_{x\in \R_{+}}\left|E\right|=0.
\end{split}
\end{equation}
\end{theorem}

As it is well known that, there have been a great number of
mathematical studies about the outflow problem,  impermeable wall
problem and  inflow problem of Navier-Stokes system, please
referring  to \cite{HfgQadcd,HQad,KNZhd,KNZhd12,HfklgQadcd,HfgQad}
and the references therein. However, to the best of our knowledge,
there are very few results about the above mentioned problems for
NSP system. Duan-Yang \cite{DY2013} firstly proved the stability of
rarefaction wave and boundary layer for outflow problem on the
two-fluid NSP system. One important point used in \cite{DY2013} is
that the large time behavior of the electric fields is trivial and
hence the two fluids indeed have the same asymptotic profiles which
are constructed from the Navier-Stokes equations without any force
under the assumptions that all physical parameters in the model must
be unit, which is obviously impractical since ions and electrons
generally have different masses. The convergence rate of
corresponding solutions toward the stationary solution was obtained
by Zhou-Li \cite{zL2013}. In the paper, we study the nonlinear
stability of the superposition of boundary layer and rarefaction
wave for outflow problem on two-fluid NSP system. The complexity of
nonlinear composite wave leads to many complicated terms in the
course of establishing the {\it a priori} estimates. Lemma
\ref{lem.V} plays crucial role to deal with the complicated terms.
Compared with Navier-Stokes system, the key to prove Theorem
\ref{main.Res.} for NSP system is to deal with the extra electric
field $E$ which is no longer $L^{2}$ integrate in space and time due
to the structure of the Poisson equation in $\eqref{3.2ab}_{5}$. The
detailed way to deal with the terms involved with electric field $E$
is stated in \eqref{I2},\eqref{zeng6y.p2} and \eqref{zeng6y.p2uiy}.
Finally, we remark that NSP system \eqref{1.1} in the
non-dimensional form depends generally on the ratios of masses,
charges and temperatures of two fluids. If we don't ignore these
physical coefficients, the two-fluid plasma system exhibits more
complex coupling structure and the corresponding analysis of the
large time behavior of solutions becomes more complicated, referring
to \cite{DLJDE} and \cite{DLjyz2}. Hence it is meaningful and
interesting to study the general physical situation for the
nonlinear stability of superposition of boundary layer and
rarefaction wave on the two-fluid NSP system in the future.

  Finally, we refer readers to \cite{DLJDE,DLjyz2,DY2013,SYZhu,zL2013}
and references therein for the study of the related works on the NSP
system. Here we would still mention several most closely related
papers: \cite{LMZ2010,ZLZ} for the spectral analysis and time-decay
of the NSP system around the constant states,
 \cite{Do,Z. Tan} for the global existence of strong solutions to the one-dimensional NSP system with large data.
 Recently, the stability of  the superposition of rarefaction wave and contact discontinuity for the
 NSP system with free boundary has been obtained by Ruan-Yin-Zhu
\cite{R-Y-Zcv}. For the investigations in the stability of the
rarefaction wave of the related models, see also \cite{DLJDE2} for
the study of the more complicated Vlasov-Poisson-Boltzmann system.

\medskip

The rest of the paper is arranged as follows.  In the main part
Section 2, we give the {\it a priori} estimates on the solutions of
the perturbative equations. The structure of  Poisson equation and
the symmetry of two-fluid system
 play important roles in the proof of the {\it a priori}
estimates. The proof of Theorem \ref{main.Res.} is concluded in
Section 3.

\medskip

\textbf{Notations:} Throughout this paper, $C$ denotes some positive
constant (generally large) and $c$ denotes  some positive constant
(generally small), where both $C$ and $c$ may take different values
in different places.
 $L^p =
L^p(\mathbb{\R_{+}}) \ (1 \leq p \leq +\infty)$ denotes
 the usual Lebesgue space on $\R_{+}$ with its norm $ \|\cdot\|_ {L^p}$, and when $p=2, +\infty$, we
write $ \| \cdot \| _{ L^2(\mathbb{\R_{+}})} = \| \cdot \|$ and $ \|
\cdot \| _{ L^\infty (\mathbb{\R_{+}})} = \| \cdot \|_{\infty}$. We
use $H^s=H^s(\R_{+})$ $(s\geq 0)$ to denote the usual Sobolev space
with respect to $x$ variable.

\section{The proof of  a priori estimates}

Let $[\rho_{i},u_{i},\rho_{e},u_{e},E]$ be the solution of the
one-dimensional two-fluid Navier-Stokes-Poisson system \eqref{1.1},
\eqref{1.a} and \eqref{1.b}.  Let $[\hat{\rho},\hat{u}]$ be the
solution of \eqref{2.12}. Now, we put the perturbation
 $[\varphi_{i},\psi_{i},\varphi_{e},\psi_{e}]$ by
\begin{eqnarray}\label{3.1}
\begin{aligned}
&\varphi_{i}=\rho_{i}-\hat{\rho}, \ \ \ \psi_{i}=u_{i}-\hat{u},
&\varphi_{e}=\rho_{e}-\hat{\rho},\ \ \psi_{e}=u_{e}-\hat{u}.
\end{aligned}
\end{eqnarray}
Then, from (\ref{1.1}) and (\ref{2.12}),
$[\varphi_{i},\psi_{i},\varphi_{e},\psi_{e}]$  satisfies
\begin{eqnarray}\label{3.2ab}
&&\left\{\begin{aligned}
& \partial_{t}\varphi_{i}+u_{i}\partial_{x}\varphi_{i}+\rho_{i}\partial_{x}\psi_{i}=-f_{i},\\
&\rho_{i}(\partial_{t}\psi_{i}+u_{i}\partial_{x}\psi_{i})+P'(\rho_{i})\partial_{x}\varphi_{i}=\partial_{x}^{2}\psi_{i}-g_{i}+\rho_{i}E,\\
&\partial_{t}\varphi_{e}+u_{e}\partial_{x}\varphi_{e}+\rho_{e}\partial_{x}\psi_{e}=-f_{e},\\
&\rho_{e}(\partial_{t}\psi_{e}+u_{e}\partial_{x}\psi_{e})+P'(\rho_{e})\partial_{x}\varphi_{e}=\partial_{x}^{2}\psi_{e}-g_{e}-\rho_{e}E,\\
&\partial_{x}E=\varphi_{i}-\varphi_{e},\ \ \ x\in
\mathbb{R}_{+},\  t>0,\\
&(\psi_{i},\psi_{e})(0,t)=0,\\
&(\phi_{i},\psi_{i},\phi_{e},\psi_{e})(x,0)\rightarrow 0, \ \
\mathrm{as} \ \ x\rightarrow +\infty,
\end{aligned}\right.
\end{eqnarray}
where $f_{\alpha},$ $g_{\alpha}$ $(\alpha=i,e)$ are the nonlinear
terms, given by
\begin{eqnarray}\label{3.3ab}
&&\left\{\begin{aligned}
&f_{\alpha}=\partial_{x}\hat{u}\varphi_{\alpha}+\partial_{x}\hat{\rho}\psi_{\alpha}+\hat{f},\\
&g_{\alpha}=\rho_{\alpha}\partial_{x}\hat{u}\psi_{\alpha}+\partial_{x}\hat{\rho}
\left[P'(\rho_{\alpha})-P'(\hat{\rho})\right]
+[\partial_{x}^{2}\hat{u}-\partial_{x}P(\hat{\rho})]\frac{\varphi_{\alpha}}{\hat{\rho}}+\hat{g}\frac{\rho_{\alpha}}{\hat{\rho}}.
\end{aligned}\right.
\end{eqnarray}

We define the solution space $X(0,T)$ by
\begin{eqnarray*}\label{3.4}
\begin{aligned}
X(0,T):=&\left\{[\varphi_{\alpha},
\psi_{\alpha},E]\in C([0,T];H^{1}),\ \ [\partial_{x}\varphi_{\alpha},\partial_{x}E]\in L^{2}([0,T];L^{2}),\right.\\
&\left.\partial_{x}\psi_{\alpha}\in L^{2}([0,T];H^{1}),\ \
\psi_{\alpha}(0,t)=0,\ \ \alpha=i,e,\ \ \forall (x,t)\in
[0,+\infty\} \times [0,T]\right\}.
\end{aligned}
\end{eqnarray*}

 The local existence of
\eqref{3.2ab} can be established by the standard iteration argument
 and hence will be skipped in the paper. To obtain
the global existence part of Theorem \ref{main.Res.}, it suffices to
prove the following Proposition \ref{priori.est} ({\it a priori}
estimates).

\begin{proposition}\label{priori.est}({\it a priori} estimates).
 Assume all the conditions listed in Theorem
$\mathrm{\ref{main.Res.}}$ hold. Let
$[\varphi_{i},\psi_{i},\varphi_{e},\psi_{e},E]$
 be a solution to the initial boundary value problem
\eqref{3.2ab} on $0\leq t\leq T$ for some positive constant T. There
exist some positive constants $C$ and $\varepsilon_{1}$ such that if
\begin{eqnarray}\label{2.5op0ui8}
\sup\limits_{0\leq t\leq
T}\left(\left\|[\varphi_{i},\psi_{i},\varphi_{e},\psi_{e}](t)\right\|_{H^1}+\|E(t)\|\right)+\epsilon+\tilde{\delta}\leq
\varepsilon_{1},
\end{eqnarray} then the solution
$[\varphi_{i},\psi_{i},\varphi_{e},\psi_{e},E]$ satisfies
\begin{equation}\label{resultuijk}
\begin{aligned}[b]
&\sup_{0\leq t\leq
T}\left\|[\varphi_{i},\psi_{i},\varphi_{e},\psi_{e},E]\right\|_{H^1}^2
+\int_{0}^{T}\|\sqrt{\pa_x\hat{u}}[\varphi_{i},\psi_{i},\varphi_{e},\psi_{e}]\|^2
dt+\int_{0}^{T}\|\pa_x[\varphi_{i},\varphi_{e},E]\|^2+\|\pa_x\left[\psi_{i},\psi_{e}\right]\|_{H^1}^2dt
\\
\leq&
C\left(\|[\varphi_{i0},\psi_{i0},\varphi_{e0},\psi_{e0}]\|_{H^{1}}^{2}+\left\|E(0,t)\right\|^2\right)+C\left(\epsilon^{\frac{1}{10}}+\tilde{\delta}^{\frac{1}{9}}\right).
\end{aligned}
\end{equation}
\end{proposition}
Using \eqref{2.5op0ui8} and the following Sobolev inequality
\begin{equation}\label{sob.}
|h(x)|\leq\sqrt{2}\|h\|^{\frac{1}{2}}\|h_{x}\|^{\frac{1}{2}}\
\textrm{for}\ h(x) \ \in H^1({\R_{+}}),
\end{equation}
we have
\begin{eqnarray}\label{2.5ploik}
\|[\varphi_{i},\psi_{i},\varphi_{e},\psi_{e}]\|_{{\infty}}\leq
\sqrt{2}\varepsilon_{1},
\end{eqnarray}
which will be repeatedly used in the following.

We prove Proposition \ref{priori.est} by  elementary energy
 methods.  Lemma \ref{lem.V} in the appendix plays a key role in the stability
analysis. Before our estimates, we should point out that the general
constant $C$ below may depend on the strength of the rarefaction
wave $\delta_{r}$ since the rarefaction wave considered here is not
weak. Now, we prove Proposition \ref{priori.est} by the following three steps.\\
\textbf{Step1:} The zero-order energy estimates.

For $\alpha=i,e,$ we define the function
$$\Phi_{\alpha}=\Phi(\rho_{\alpha},\hat{\rho})=\int_{\hat{\rho}}^{\rho_{\alpha}}\frac{P(s)-P(\hat{\rho})}{s^{2}}ds$$
and
$\eta_\al=\rho_{\alpha}\Phi_{\alpha}+\frac{1}{2}\rho_{\alpha}\psi_{\alpha}^{2}$.
Direct calculations give rise to
\begin{eqnarray}\label{4.5}
\begin{aligned}[b]
&\partial_{t}\eta_i+
\partial_{x}\left[u_{i}\eta_i+(P(\rho_{i})-P(\hat{\rho}))\psi_{i}-\psi_{i}\partial_{x}\psi_{i}\right]
+\partial_{x}\hat{u}\left[P(\rho_{i})-P(\hat{\rho})-P'(\hat{\rho})\varphi_{i}+\rho_{i}\psi_{i}^{2}\right]\\&+(\partial_{x}\psi_{i})^{2}
=\rho_{i}\psi_{i}E-\partial_{x}^{2}\hat{u}\frac{\varphi_{i}\psi_{i}}{\hat{\rho}}-\hat{g}\frac{\rho_{i}\psi_{i}}{\hat{\rho}}-P'(\hat{\rho})
\hat{f}\frac{\varphi_{i}}{\hat{\rho}}
\end{aligned}
\end{eqnarray}
and
\begin{eqnarray}\label{4.5bg}
\begin{aligned}[b]
&\partial_{t}\eta_e+
\partial_{x}\left[u_{e}\eta_e+(P(\rho_{e})-P(\hat{\rho}))\psi_{e}-\psi_{e}\partial_{x}\psi_{e}\right]
+\partial_{x}\hat{u}\left[P(\rho_{e})-P(\hat{\rho})-P'(\hat{\rho})\varphi_{e}+\rho_{e}\psi_{e}^{2}\right]\\&+(\partial_{x}\psi_{e})^{2}
=-\rho_{e}\psi_{e}E-\partial_{x}^{2}\hat{u}\frac{\varphi_{e}\psi_{e}}{\hat{\rho}}-\hat{g}\frac{\rho_{e}\psi_{e}}{\hat{\rho}}-P'(\hat{\rho})
\hat{f}\frac{\varphi_{e}}{\hat{\rho}}.
\end{aligned}
\end{eqnarray}

Taking the summation of \eqref{4.5} and \eqref{4.5bg}, and
integrating the resulting equation with respect to $x$ over $
\R_{+}$, we arrive at
\begin{eqnarray}\label{4.5vbbgf}
\begin{aligned}[b]
&\frac{d}{dt}\int_{\mathbb{R}_{+}}(\eta_i+\eta_e)dx
+|u_{b}|\left[(\rho_{i} \Phi_{i})(0,t)+(\rho_{e} \Phi_{e})(0,t)\right]+\int_{\mathbb{R}_{+}}\left[(\partial_{x}\psi_{i})^{2}+(\partial_{x}\psi_{e})^{2}\right]dx\\
&+\int_{\mathbb{R}_{+}}\partial_{x}\hat{u}\left[P(\rho_{i})-P(\hat{\rho})-P'(\hat{\rho})\varphi_{i}
+P(\rho_{e})-P(\hat{\rho})-P'(\hat{\rho})\varphi_{e}
+\rho_{i}\psi_{i}^{2}+\rho_{e}\psi_{e}^{2}\right]dx
\\=&\underbrace{\int_{\mathbb{R}_{+}}(\rho_{i}\psi_{i}-\rho_{e}\psi_{e})E
dx}_{I_{1}}
\underbrace{-\int_{\mathbb{R}_{+}}\left(\partial_{x}^{2}\hat{u}\frac{\varphi_{i}\psi_{i}}{\hat{\rho}}+\partial_{x}^{2}\hat{u}\frac{\varphi_{e}\psi_{e}}{\hat{\rho}}\right)dx}_{Q_{1}}
\\ &\underbrace{-\int_{\mathbb{R}_{+}}\left(\hat{g}\frac{\rho_{i}\psi_{i}}{\hat{\rho}}+\hat{g}\frac{\rho_{e}\psi_{e}}{\hat{\rho}}\right)dx}_{Q_{2}}
\underbrace{-\int_{\mathbb{R}_{+}}\left(P'(\hat{\rho})\hat{f}\frac{\varphi_{i}}{\hat{\rho}}+P'(\hat{\rho})
\hat{f}\frac{\varphi_{e}}{\hat{\rho}}\right)dx}_{Q_{3}}.
\end{aligned}
\end{eqnarray}
Here we have used the boundary condition $\eqref{3.2ab}_{6}$ and
$u_{b}<0.$

From Poisson equation and mass conservation equation, we have
\begin{eqnarray}\label{4.vb5bg}
\begin{aligned}[b]
\pa_xE=\rho_{i}-\rho_{e}, \ \ \pa_tE=\rho_{e}u_{e}-\rho_{i}u_{i}.
\end{aligned}
\end{eqnarray}
 Now we mainly make use of \eqref{4.vb5bg} to deal with the difficult term $I_{1}$. Then one has by
integration by parts
\begin{equation}\label{I2}
\begin{split}
I_1=&\int_{\mathbb{R}_{+}}E(\rho_{i}u_{i}-\rho_{e}u_{e})dx-\int_{\mathbb{R}_{+}}E(\rho_{i}-\rho_{e})\tilde{u}dx\\
=&-\frac{1}{2}\frac{d}{dt}\int_{\mathbb{R}_{+}}E^{2}dx-\frac{|u_{b}|}{2}E^{2}(0,t)\underbrace{+\frac{1}{2}\int_{\mathbb{R}_{+}}\pa_{x}\hat{u}E^{2}dx}_{I_{2}}.
\end{split}
\end{equation}

Notice that
$\partial_{x}\hat{u}=\partial_{x}\tilde{u}+\partial_{x}u^{r_{2}}\geq0$
from $\partial_{x}\tilde{u}\geq0$ and $\partial_{x}u^{r_{2}}\geq0$.
Now we pay our attention on the bad term $I_{2}$ since the electric
field $E$ is no longer $L^{2}$ integrate in space and time due to
the structure of the Poisson equation. The main idea is to make use
of the good term
$$\int_{\R_{+}}\partial_{x}\hat{u}\left[
\rho_{i}\psi_{i}^{2}+\rho_{e}\psi_{e}^{2}\right]dx$$ to absorb
$I_{2}$. For this, multiplying $\eqref{3.2ab}_{2}$ and
$\eqref{3.2ab}_{4}$ by $\frac{1}{4\rho_{i}}E\pa_{x}\hat{u}$ and
$-\frac{1}{4\rho_{e}}E\pa_{x}\hat{u}$ respectively, then integrating
the resulting equations over $\R_{+}$ and taking the summation of
the resulting equations, one has

\begin{eqnarray}\label{zeng6y.p2}
\begin{aligned}[b]
I_{2}=&\frac{1
}{4}\frac{d}{dt}\int_{\R_{+}}\pa_{x}\hat{u}(\psi_{i}-\psi_{e})E dx
\underbrace{-\frac{1}{4}\int_{\R_{+}}\left(\psi_{i}-\psi_{e}\right)\pa_tE
\pa_{x}\hat{u}dx}_{I_{3}}\underbrace{-\frac{1}{4}\int_{\R_{+}}\left(\psi_{i}-
\psi_{e}\right)E \pa_t\pa_x
\hat{u}dx}_{I_{4}}\\
&\underbrace{+\frac{1}{4}\int_{\R_{+}}\left(u_{i}\pa_{x}\psi_{i}-u_{e}\pa_{x}\psi_{e}\right)E\pa_x
\hat{u}dx}_{I_{5}}
\underbrace{-\frac{1}{4}\int_{\R_{+}}\left(\frac{\pa_{x}^{2}\psi_{i}}{\rho_{i}}-
\frac{\pa_{x}^{2}\psi_{e}}{\rho_{e}}\right)E\pa_x \hat{u}dx}_{I_{6}}
\\ & \underbrace{+\frac{1}{4}\int_{\R_{+}}\left(\frac{P'(\rho_{i})}{\rho_{i}}\pa_{x}\varphi_{i}-
\frac{P'(\rho_{e})}{\rho_{e}}\pa_{x}\varphi_{e}\right)E\pa_x
\hat{u}dx}_{I_{7}}\underbrace{+\frac{1}{4}\int_{\R_{+}}\left(\frac{
g_{i}}{\rho_{i}}- \frac{ g_{e}}{\rho_{e}}\right)E\pa_x
\hat{u}dx}_{I_{8}}.
\end{aligned}
\end{eqnarray}

From \eqref{4.vb5bg}, we have
\begin{eqnarray}\label{zeng6y.pbg2}
\begin{aligned}[b]
\pa_tE=\rho_{e}\psi_{e}-\rho_{i}\psi_{i}+(\varphi_{e}-\varphi_{i})\hat{u}.
\end{aligned}
\end{eqnarray}
Then we make use of \eqref{zeng6y.pbg2} to deal with the difficult
term $I_{3}$. Therefore, one has
\begin{eqnarray}\label{zeng6y.p2uiy}
\begin{aligned}[b]
I_{3}=&\underbrace{\frac{1}{4}\int_{\R_{+}}\pa_x \hat{u}(
\rho_{i}\psi_{i}^{2}+\rho_{e}\psi_{e}^{2})dx}_{I_{9}}
\underbrace{-\frac{1}{4}\int_{\R_{+}}\pa_x
\hat{u}(\rho_{e}+\rho_{i})\psi_{i}\psi_{e}dx}_{I_{10}}
\underbrace{+\frac{1}{4}\int_{\R_{+}}\pa_x
\hat{u}(\psi_{e}-\psi_{i})(\varphi_{e}-\varphi_{i})\hat{u}dx}_{I_{11}}.
\end{aligned}
\end{eqnarray}

Combining \eqref{4.5vbbgf}-\eqref{zeng6y.p2uiy}, we arrive at the
following equality
\begin{eqnarray}\label{4.5vbbg}
\begin{aligned}[b]
&\frac{d}{dt}\int_{\mathbb{R}_{+}}\left(\eta_i+\eta_e+\frac{E^{2}}{2}\right)dx-\frac{1
}{4}\frac{d}{dt}\int_{\R_{+}}\pa_{x}\hat{u}(\psi_{i}-\psi_{e})E
dx+|u_{b}|\bigg[\rho_{i} \Phi_{i}(0,t)+\rho_{e}
\Phi_{e}(0,t)+\frac{E^{2}}{2}(0,t)
\bigg]\\&+\int_{\mathbb{R}_{+}}\left[(\partial_{x}\psi_{i})^{2}+(\partial_{x}\psi_{e})^{2}\right]dx+\int_{\mathbb{R}_{+}}\partial_{x}\hat{u}\left[P(\rho_{i})-P(\hat{\rho})-P'(\hat{\rho})\varphi_{i}
+P(\rho_{e})-P(\hat{\rho})-P'(\hat{\rho})\varphi_{e}\right]dx
\\&+\left[\int_{\mathbb{R}_{+}}\partial_{x}\hat{u}\left[\rho_{i}\psi_{i}^{2}+\rho_{e}\psi_{e}^{2}\right]dx-I_{9}-I_{10}\right]
=\sum_{i=1}^{3}Q_{i}+\sum_{i=4}^{8}I_{i}+I_{11}.
\end{aligned}
\end{eqnarray}

First of all, we use \eqref{2.5ploik} to deal with the left terms in
\eqref{4.5vbbg} as follows:
\begin{equation*}\label{wednm}
\begin{aligned}[b]
|u_{b}|\bigg[\rho_{i} \Phi_{i}(0,t)+\rho_{e}
\Phi_{e}(0,t)+\frac{E^{2}}{2}(0,t) \bigg] \geq c
\left[\varphi_{i}^{2}(0,t)+\varphi_{e}^{2}(0,t)+E^{2}(0,t)\right],
\end{aligned}
\end{equation*}
\begin{equation*}\label{wed}
\begin{aligned}[b]
&\int_{\mathbb{R}_{+}}\partial_{x}\hat{u}\left[P(\rho_{i})-P(\hat{\rho})-P'(\hat{\rho})\varphi_{i}
+P(\rho_{e})-P(\hat{\rho})-P'(\hat{\rho})\varphi_{e}\right]dx \geq
c\|\sqrt{\partial_{x}\hat{u}}[\varphi_{i},\varphi_{e}]\|^2,
\end{aligned}
\end{equation*}
and
\begin{equation*}\label{wed}
\begin{aligned}[b]
&\int_{\mathbb{R}_{+}}\partial_{x}\hat{u}\left[\rho_{i}\psi_{i}^{2}+\rho_{e}\psi_{e}^{2}\right]dx-I_{9}-I_{10}
=\int_{\mathbb{R}_{+}}\partial_{x}\hat{u}\left[\frac{3}{4}\rho_{i}\psi_{i}^{2}
+\frac{1}{4}\left( \rho_{e}+
\rho_{i}\right)\psi_{i}\psi_{e}+\frac{3}{4}\rho_{e}\psi_{e}^{2}\right]dx\\
\geq&\frac{1}{4}\int_{\mathbb{R}_{+}}\hat{\rho}\partial_{x}\hat{u}\left[3\psi_{i}^{2}
+2\psi_{i}\psi_{e}+3\psi_{e}^{2}\right]dx-
C\|[\varphi_{i},\varphi_{e}]\|_{\infty}\|\sqrt{\partial_{x}\hat{u}}[\psi_{i},\psi_{e}]\|^2\\
\geq&\frac{1}{4}\int_{\mathbb{R}_{+}}\hat{\rho}\partial_{x}\hat{u}\left[2(\psi_{i}^{2}+\psi_{e}^{2})+(\psi_{i}+\psi_{e})^{2}
\right]dx-
C\varepsilon_{1}\|\sqrt{\partial_{x}\hat{u}}[\psi_{i},\psi_{e}]\|^2.
\end{aligned}
\end{equation*}
 Therefore, we have
\begin{equation*}\label{2.3ghj}
\begin{split}
&\int_{\mathbb{R}_{+}}\partial_{x}\hat{u}\left[\rho_{i}\psi_{i}^{2}+\rho_{e}\psi_{e}^{2}\right]dx-I_{9}-I_{10}
\geq c\|\sqrt{\partial_{x}\hat{u}}[\psi_{i},\psi_{e}]\|^2,
\end{split}
\end{equation*}
where we take $\varepsilon_{1}$ small enough.

 Before our estimates, we take
$q=10$ and $\theta=\frac{1}{8}$ in the following for brevity.
 By employing \eqref{2.5ploik},  \eqref{2.5op0ui8},
 $\eqref{3.3ab}_{2}$,
$\eqref{3.2ab}_{5}$, Lemma \ref{cl.Re.Re2.}, Lemma \ref{lem.V},
Young inequality, Cauchy-Schwarz's inequality with $0<\eta<1$,
Sobolev inequality \eqref{sob.}, the boundary condition
$\psi_{i}(0,t)=\psi_{e}(0,t)=0$ and integrating by parts, we obtain
the estimates on the right terms in \eqref{4.5vbbg} as follows:
\begin{eqnarray*}\label{4.11}
\begin{aligned}
&|Q_{1}|+|Q_{2}|+|Q_{3}|\\ \leq &
C\|[\varphi_{i},\varphi_{e},\psi_{i},\psi_{e}]\|_{\infty}\int_{\mathbb{R}_{+}}
\left(|\hat{f}|+|\hat{g}+\partial_{x}^{2}u^{r_{2}}|+|\partial_{x}^{2}u^{r_{2}}|\right)
dx+C\int_{\mathbb{R}_{+}}|\partial_{x}^{2}\tilde{u}|
\left(\varphi_{i}^{2}+\varphi_{e}^{2}+\psi_{i}^{2}+\psi_{e}^{2}\right)
dx\\
\leq
&C\|[\varphi_{i},\varphi_{e},\psi_{i},\psi_{e}]\|^{\frac{1}{2}}\|\partial_{x}[\varphi_{i},\varphi_{e},\psi_{i},\psi_{e}]\|^{\frac{1}{2}}
\left[\frac{\tilde{\delta}}{1+\tilde{\delta}t}+\epsilon^{\theta}(1+t)^{-(1-\theta)}\ln(1+\tilde{\delta}t)+\epsilon^{\frac{1}{q}}(1+t)^{-1+\frac{1}{q}}\right]\\
&+C\tilde{\delta}^{2}\left[\varphi_{i}^{2}(0,t)+\varphi_{e}^{2}(0,t)\right]+C\tilde{\delta}\|\partial_{x}[\varphi_{i},\varphi_{e},\psi_{i},\psi_{e}]\|^{2}\\
\leq
&C(\tilde{\delta}^{\frac{2}{3}}+\epsilon^{\frac{1}{10}})\|\partial_{x}[\varphi_{i},\varphi_{e},\psi_{i},\psi_{e}]\|^{2}
+C\frac{\tilde{\delta}^{\frac{10}{9}}}{(1+\tilde{\delta}t)^{\frac{4}{3}}}+C\epsilon^{\frac{1}{10}}(1+t)^{-\frac{13}{12}}
+C\tilde{\delta}^{2}\left[\varphi_{i}^{2}(0,t)+\varphi_{e}^{2}(0,t)\right],
\end{aligned}
\end{eqnarray*}

\begin{eqnarray*}\label{4.1p}
\begin{aligned}
|I_{4}|\leq& C \|[\psi_{i},\psi_{e}]\|_{\infty}\|E\|_{\infty}
\|[\partial_{t}\partial_{x}u^{r_{2}}]\|_{L^{1}}\\
\leq
&C\epsilon^{\frac{1}{q}}(1+t)^{-1+\frac{1}{q}}\|[\psi_{i},\psi_{e}]\|^{\frac{1}{2}}
\|\partial_{x}[\psi_{i},\psi_{e}]\|^{\frac{1}{2}}
\|E\|^{\frac{1}{2}}\|\partial_{x}E\|^{\frac{1}{2}}\\
\leq
&C\epsilon^{\frac{1}{10}}(1+t)^{-\frac{9}{5}}+C\epsilon^{\frac{1}{10}}
\|\partial_{x}[\psi_{i},\psi_{e},E]\|^{2},
\end{aligned}
\end{eqnarray*}
\begin{eqnarray*}\label{4.1pjn}
\begin{aligned}
&|I_{5}|+|I_{6}|+|I_{7}\\ \leq&
C\|\partial_{x}u^{r_{2}}\|_{\infty}\|E\|\|\partial_{x}[\psi_{i},\psi_{e},\partial_{x}\psi_{i},\partial_{x}\psi_{e},\varphi_{i},\varphi_{e}]\|
\\&+C\int_{\mathbb{R}_{+}}\left|\partial_{x}[\psi_{i},\psi_{e},\partial_{x}\psi_{i},\partial_{x}\psi_{e},\varphi_{i},\varphi_{e}]\right||E|\partial_{x}\tilde{u}dx\\
\leq
&C\epsilon^{\theta}(1+t)^{-(1-\theta)}\|E\|\|\partial_{x}[\psi_{i},\psi_{e},\partial_{x}\psi_{i},\partial_{x}\psi_{e},\varphi_{i},\varphi_{e}]\|\\&+C
\tilde{\delta}\|\partial_{x}[\psi_{i},\psi_{e},\partial_{x}\psi_{i},\partial_{x}\psi_{e},\varphi_{i},\varphi_{e},E]\|^{2}+C\tilde{\delta}^{2}E^{2}(0,t)\\
\leq &C
(\tilde{\delta}+\epsilon^{\frac{1}{8}})\|\partial_{x}[\psi_{i},\psi_{e},\partial_{x}\psi_{i},\partial_{x}\psi_{e},\varphi_{i},\varphi_{e},E]\|^{2}+C\tilde{\delta}^{2}E^{2}(0,t)
+C\epsilon^{\frac{1}{8}}(1+t)^{-\frac{7}{4}},
\end{aligned}
\end{eqnarray*}
\begin{eqnarray*}\label{4.1pjfn}
\begin{aligned}
|I_{8}| \leq&
C\int_{\mathbb{R}_{+}}\left(|g_{i}|+|g_{e}|\right)|E||\partial_{x}\hat{u}|dx\\
\leq
&C\int_{\mathbb{R}_{+}}\left\{|\partial_{x}\hat{u}|(|\psi_{i}|+|\psi_{e}|+|\varphi_{i}|+|\varphi_{e}|)+|\partial_{x}^{2}\hat{u}|(|\varphi_{i}|+|\varphi_{e}|)\right\}|E||\partial_{x}\hat{u}|dx
+C\int_{\mathbb{R}_{+}}|\hat{g}||E||\partial_{x}\hat{u}|dx\\
\leq
&C\int_{\mathbb{R}_{+}}\left(|\partial_{x}\tilde{u}|^{2}+|\partial_{x}^{2}\tilde{u}|
+|\partial_{x}u^{r_{2}}|^{2}+|\partial_{x}^{2}u^{r_{2}}|\right)\left(\psi_{i}^{2}+\psi_{e}^{2}+\varphi_{i}^{2}+\varphi_{e}^{2}+E^{2}\right)dx
+C\int_{\mathbb{R}_{+}}|\hat{g}|^{2}dx
\\
\leq
&C\tilde{\delta}^{2}\left[\varphi_{i}^{2}(0,t)+\varphi_{e}^{2}(0,t)+E^{2}(0,t)\right]+C\tilde{\delta}\|\partial_{x}[\varphi_{i},\varphi_{e},\psi_{i},\psi_{e},E]\|^{2}
+C\epsilon^{1+\frac{2}{q}}(1+t)^{-2(1-\frac{1}{q})}+C\tilde{\delta}(1+t)^{-2}\\
&+C\left(\|[\varphi_{i},\varphi_{e},\psi_{i},\psi_{e},E]\|\|\partial_{x}[\varphi_{i},\varphi_{e},\psi_{i},\psi_{e},E]\|\right)
\left(\|\partial_{x}^{2}u^{r_{2}}\|_{L^{1}}+\|\partial_{x}u^{r_{2}}\|^{2}\right)
\\
\leq
&C(\tilde{\delta}+\epsilon^{\frac{1}{10}})\|\partial_{x}[\varphi_{i},\varphi_{e},\psi_{i},\psi_{e},E]\|^{2}
+C(\epsilon^{\frac{1}{10}}+\tilde{\delta})(1+t)^{-\frac{9}{5}}
+C\tilde{\delta}^{2}\left[\varphi_{i}^{2}(0,t)+\varphi_{e}^{2}(0,t)+E^{2}(0,t)\right]
\end{aligned}
\end{eqnarray*}
and
\begin{eqnarray*}\label{4.1pjbvdn}
\begin{aligned}
|I_{11}| \leq&
C\int_{\mathbb{R}_{+}}|\partial_{x}\tilde{u}|\left(|\psi_{i}|+|\psi_{e}|\right)|\partial_{x}E|dx
+C\int_{\mathbb{R}_{+}}|\partial_{x}u^{r_{2}}|\left(|\psi_{i}|+|\psi_{e}|\right)|\partial_{x}E|dx\\
\leq&\eta\|\partial_{x}E\|^{2}+C_{\eta}\int_{\mathbb{R}_{+}}|\partial_{x}\tilde{u}|^{2}\left(|\psi_{i}|^{2}+|\psi_{e}|^{2}\right)dx
+C\|\partial_{x}u^{r_{2}}\|_{\infty}(\|\psi_{i}\|+\|\psi_{e}\|)\|\partial_{x}E\|\\
\leq&\eta\|\partial_{x}E\|^{2}+C_{\eta}\tilde{\delta}^{2}\|\partial_{x}[\psi_{i},\psi_{e}]\|^{2}
+C\epsilon^{\theta}(1+t)^{-2(1-\theta)}(\|\psi_{i}\|^{2}+\|\psi_{e}\|^{2})+C\epsilon^{\theta}\|\partial_{x}E\|^{2}
\\
\leq&(\eta+C\epsilon^{\frac{1}{8}})\|\partial_{x}E\|^{2}+C_{\eta}\tilde{\delta}^{2}\|\partial_{x}[\psi_{i},\psi_{e}]\|^{2}
+C\epsilon^{\frac{1}{8}}(1+t)^{-\frac{7}{4}}.
\end{aligned}
\end{eqnarray*}

Substituting  the estimates above into \eqref{4.5vbbg} and
integrating the resulting inequality over $[0,T]$ and using Cauchy
Schwarz's inequality, and taking $\epsilon$, $\tilde{\delta}$ and
$\varepsilon_{1}$ small enough, one can see that
\begin{eqnarray}\label{4.28de}
\begin{aligned}[t]
&\|[\varphi_{i},\varphi_{e},\psi_{i},\psi_{e},E]\|^{2}+\int_{0}^{T}\left[\|\partial_{x}[\psi_{i},\psi_{e}]\|^{2}
+\|\sqrt{\partial_{x}\hat{u}}[\varphi_{i},\varphi_{e},\psi_{i},\psi_{e}]\|^2\right]dt\\&
+\int_{0}^{T}\left[(\varphi_{i})^{2}(0,t)+(\varphi_{e})^{2}(0,t)+E^{2}(0,t)\right]dt\\
\leq&C\left(\|[\varphi_{i0},\varphi_{e0},\psi_{i0},\psi_{e0}]\|^{2}+\|E(x,0)\|^{2}\right)
+(\eta+C\epsilon^{\frac{1}{10}}+C\tilde{\delta}^{\frac{2}{3}})\|\partial_{x}
[\varphi_{i},\varphi_{e},\partial_{x}\psi_{i},\partial_{x}\psi_{e},E]\|^{2}
+C(\epsilon^{\frac{1}{10}}+\tilde{\delta}^{\frac{1}{9}}).
\end{aligned}
\end{eqnarray}
 {\bf Step 2.} {\it Dissipation of
$\pa_x[\varphi_{i},\varphi_{e},E]$.}

We first differentiate  $\eqref{3.2ab}_{1}$ and $\eqref{3.2ab}_{3}$
with respect to $x$, respectively, to obtain
\begin{equation}\label{d.rho}
\pa_t\pa_x\varphi_{i}+\pa_xu_{i}\pa_x\varphi_{i}+u_{i}\pa^{2}_x\varphi_{i}+\pa_x\rho_{i}\pa_x\psi_{i}
+\rho_{i}\pa^{2}_x\psi_{i}+\pa^{2}_x\hat{u}\varphi_{i}+\pa_x\hat{u}\pa_x\varphi_{i}+\pa_x\hat{\rho}\pa_x\psi_{i}
+\pa^{2}_x\hat{\rho}\psi_{i}+\pa_x\hat{f}=0
\end{equation}
and
\begin{equation}\label{d.rhode}
\pa_t\pa_x\varphi_{e}+\pa_xu_{e}\pa_x\varphi_{e}+u_{e}\pa^{2}_x\varphi_{e}+\pa_x\rho_{e}\pa_x\psi_{e}
+\rho_{e}\pa^{2}_x\psi_{e}+\pa^{2}_x\hat{u}\varphi_{e}+\pa_x\hat{u}\pa_x\varphi_{e}+\pa_x\hat{\rho}\pa_x\psi_{e}
+\pa^{2}_x\hat{\rho}\psi_{e}+\pa_x\hat{f}=0.
\end{equation}
Then multiplying $\eqref{3.2ab}_{5}$, $\eqref{3.2ab}_{2}$,
$\eqref{3.2ab}_{4}$, \eqref{d.rho} and \eqref{d.rhode} by $\pa_xE$,
 $\frac{\pa_x\varphi_{i}}{\rho_{i}}$,
$\frac{\pa_x\varphi_{e}}{\rho_{e}}$,
$\frac{\pa_x\varphi_{i}}{\rho_{i}^{2}}$ and
$\frac{\pa_x\varphi_{e}}{\rho_{e}^{2}}$, and integrating the
resulting equalities over $\R_{+}$, one has
\begin{equation*}\label{d.phi.ip}
\begin{split}
\int_{\R_{+}}(\pa_xE)^{2}dx
=[\varphi_{e}(0,t)-\varphi_{i}(0,t)]E(0,t)-\int_{\R_{+}}\pa_x(\varphi_{i}-\varphi_{e})E
dx,
\end{split}
\end{equation*}

\begin{equation*}\label{tu1.ip1}
\begin{aligned}[b]
&\int_{\R_{+}}\pa_t\psi_{i}\pa_x\varphi_{i}
dx+\int_{\R_{+}}u_{i}\pa_x\psi_{i} \pa_x\varphi_{i} dx
+\int_{\R_{+}}\frac{P'(\rho_{i})}{\rho_{i}}(\pa_x\varphi_{i})^{2}
dx\\
&\qquad=\int_{\R_{+}}\pa_x\varphi_{i}E dx+\int_{\R_{+}}\pa^2_x
\psi_{i}\frac{\pa_x\varphi_{i}}{\rho_{i}}
dx-\int_{\R_{+}}g_{i}\frac{\pa_x\varphi_{i}}{\rho_{i}}dx,
\end{aligned}
\end{equation*}

\begin{equation*}\label{tu1.ip2}
\begin{aligned}[b]
&\int_{\R_{+}}\pa_t\psi_{e}\pa_x\varphi_{e}
dx+\int_{\R_{+}}u_{e}\pa_x\psi_{e} \pa_x\varphi_{e} dx
+\int_{\R_{+}}\frac{P'(\rho_{e})}{\rho_{e}}(\pa_x\varphi_{e})^{2}
dx\\
&\qquad=-\int_{\R_{+}}\pa_x\varphi_{e}E dx+\int_{\R_{+}}\pa^2_x
\psi_{e}\frac{\pa_x\varphi_{e}}{\rho_{e}}
dx-\int_{\R_{+}}g_{e}\frac{\pa_x\varphi_{e}}{\rho_{e}}dx,
\end{aligned}
\end{equation*}

\begin{equation*}\label{d.rho.ip3}
\begin{aligned}[b]
&\int_{\R_{+}}\frac{\pa_x\varphi_{i}}{\rho_{i}^{2}}\pa_t\pa_x\varphi_{i}
dx+\int_{\R_{+}}\pa_xu_{i}\frac{(\pa_x\varphi_{i})^{2}}{\rho_{i}^{2}}dx+\int_{\R_{+}}
u_{i}\frac{\pa_x\varphi_{i}\pa^{2}_x\varphi_{i}}{\rho_{i}^{2}}dx\\
&\qquad+\int_{\R_{+}}
\frac{\pa_x\varphi_{i}}{\rho_{i}^{2}}\pa_x\rho_{i}\pa_x\psi_{i} dx
+\int_{\R_{+}}\pa_x\hat{u}\frac{(\pa_x\varphi_{i})^{2}}{\rho_{i}^{2}}dx
\\
&\qquad=-\int_{\R_{+}}\pa^2_x
\psi_{i}\frac{\pa_x\varphi_{i}}{\rho_{i}}dx
-\int_{\R_{+}}\pa^{2}_x\hat{u}\varphi_{i}\frac{\pa_x\varphi_{i}}{\rho_{i}^{2}}dx-
\int_{\R_{+}}\pa_x\hat{\rho}\pa_x\psi_{i}\frac{\pa_x\varphi_{i}}{\rho_{i}^{2}}dx\\
&\qquad
-\int_{\R_{+}}\pa^{2}_x\hat{\rho}\psi_{i}\frac{\pa_x\varphi_{i}}
{\rho_{i}^{2}}dx-\int_{\R_{+}}\pa_x\hat{f}\frac{\pa_x\varphi_{i}}
{\rho_{i}^{2}}dx
\end{aligned}
\end{equation*}
and
\begin{equation*}\label{d.rho.ip4}
\begin{aligned}[b]
&\int_{\R_{+}}\frac{\pa_x\varphi_{e}}{\rho_{e}^{2}}\pa_t\pa_x\varphi_{e}
dx+\int_{\R_{+}}\pa_xu_{e}\frac{(\pa_x\varphi_{e})^{2}}{\rho_{e}^{2}}dx+\int_{\R_{+}}
u_{e}\frac{\pa_x\varphi_{e}\pa^{2}_x\varphi_{e}}{\rho_{e}^{2}}dx
\\
&\qquad+\int_{\R_{+}}
\frac{\pa_x\varphi_{e}}{\rho_{e}^{2}}\pa_x\rho_{e}\pa_x\psi_{e}
dx+\int_{\R_{+}}\pa_x\hat{u}\frac{(\pa_x\varphi_{e})^{2}}{\rho_{e}^{2}}dx\\
&\qquad=-\int_{\R_{+}}\pa^2_x
\psi_{e}\frac{\pa_x\varphi_{e}}{\rho_{e}}dx
-\int_{\R_{+}}\pa^{2}_x\hat{u}\varphi_{e}\frac{\pa_x\varphi_{e}}{\rho_{e}^{2}}dx-
\int_{\R_{+}}\pa_x\hat{\rho}\pa_x\psi_{e}\frac{\pa_x\varphi_{e}}{\rho_{e}^{2}}dx\\
&\qquad
-\int_{\R_{+}}\pa^{2}_x\hat{\rho}\psi_{e}\frac{\pa_x\varphi_{e}}{\rho_{e}^{2}}dx-\int_{\R_{+}}\pa_x\hat{f}\frac{\pa_x\varphi_{e}}
{\rho_{e}^{2}}dx.
\end{aligned}
\end{equation*}

The summation of the equalities above further implies
\begin{equation}\label{sum.d1}
\begin{aligned}[b]
\frac{d}{dt}&\int_{\R_{+}}\left(\psi_{i} \pa_x\varphi_{i}+\psi_{e}
\pa_x\varphi_{e} \right)dx
+\frac{d}{dt}\int_{\R_{+}}\left(\frac{1}{2\rho_{i}^{2}}(\pa_x\varphi_{i})^{2}
+\frac{1}{2\rho_{e}^{2}}(\pa_x\varphi_{e})^{2}\right)dx\\
&+\int_{\R_{+}}\left[\pa_x\hat{u}\frac{(\pa_x\varphi_{i})^{2}}{\rho_{i}^{2}}
+\pa_x\hat{u}\frac{(\pa_x\varphi_{e})^{2}}{\rho_{e}^{2}}+\frac{P'(\rho_{i})}{\rho_{i}}(\pa_x\varphi_{i})^{2}
+\frac{P'(\rho_{e})}{\rho_{e}}(\pa_x\varphi_{e})^{2}+(\pa_xE)^{2}\right]dx\\
=&[\varphi_{e}(0,t)-\varphi_{i}(0,t)]E(0,t)+\int_{\R_{+}}\left(\psi_{i}
\pa_t\pa_x\varphi_{i}+\psi_{e} \pa_t\pa_x\varphi_{e}\right) dx
\\
&
-\int_{\R_{+}}\left((\pa_x\varphi_{i})^{2}\rho_{i}^{-3}\partial_{t}\rho_{i}
+(\pa_x\varphi_{e})^{2}\rho_{e}^{-3}\partial_{t}\rho_{e}
\right)dx-\int_{\R_{+}}\left(u_{i}\pa_x\psi_{i}
\pa_x\varphi_{i}+u_{e}\pa_x\psi_{e}
\pa_x\varphi_{e}\right) dx\\
&-\int_{\R_{+}}\left(g_{i}\frac{\pa_x\varphi_{i}}{\rho_{i}}+g_{e}\frac{\pa_x\varphi_{e}}{\rho_{e}}\right)dx
-\int_{\R_{+}}\left(\pa_xu_{i}\frac{(\pa_x\varphi_{i})^{2}}{\rho_{i}^{2}}+\pa_xu_{e}\frac{(\pa_x\varphi_{e})^{2}}{\rho_{e}^{2}}\right)dx\\
&-\int_{\R_{+}}\left(u_{i}\frac{\pa_x\varphi_{i}\pa^{2}_x\varphi_{i}}{\rho_{i}^{2}}+
u_{e}\frac{\pa_x\varphi_{e}\pa^{2}_x\varphi_{e}}{\rho_{e}^{2}}\right)
dx-\int_{\R_{+}}\left(
\frac{\pa_x\varphi_{i}}{\rho_{i}^{2}}\pa_x\rho_{i}\pa_x\psi_{i}+
\frac{\pa_x\varphi_{e}}{\rho_{e}^{2}}\pa_x\rho_{e}\pa_x\psi_{e}\right)
dx\\
&-\int_{\R_{+}}\left(\pa^{2}_x\hat{u}\varphi_{i}\frac{\pa_x\varphi_{i}}{\rho_{i}^{2}}
+\pa^{2}_x\hat{u}\varphi_{e}\frac{\pa_x\varphi_{e}}{\rho_{e}^{2}}\right)
dx-
\int_{\R_{+}}\left(\pa_x\hat{\rho}\pa_x\psi_{i}\frac{\pa_x\varphi_{i}}{\rho_{i}^{2}}
+\pa_x\hat{\rho}\pa_x\psi_{e}\frac{\pa_x\varphi_{e}}{\rho_{e}^{2}}\right)
dx\\
&-\int_{\R_{+}}\left(\pa^{2}_x\hat{\rho}\psi_{i}\frac{\pa_x\varphi_{i}}{\rho_{i}^{2}}
+\pa^{2}_x\hat{\rho}\psi_{e}\frac{\pa_x\varphi_{e}}{\rho_{e}^{2}}\right)
dx -\int_{\R_{+}}\left(\pa_x\hat{f}\frac{\pa_x\varphi_{i}}
{\rho_{i}^{2}}+\pa_x\hat{f}\frac{\pa_x\varphi_{e}}
{\rho_{e}^{2}}\right)dx=\sum\limits_{l=1}^{12}J_{l},
\end{aligned}
\end{equation}
where $J_l$ $(1\leq l\leq 12)$  denote the corresponding terms on
the left hand side of \eqref{sum.d1}.

 Notice the fact that
$|\partial_{x}\hat{n}|\leq C\partial_{x}\hat{u} $,
$|\partial^{2}_{x}\hat{n}|\leq
C(|\partial^{2}_{x}\hat{u}|+|\partial_{x}\hat{u}|^{2})$ and
$u_{b}<0$. We now turn to estimate $J_l$ $(1\leq l\leq 12)$ term by
term.
 By applying Holder inequality, Cauchy-Schwarz's inequality
with $0<\eta<1$, Sobolev inequality \eqref{sob.},  Lemma
\ref{cl.Re.Re2.}, Lemma \ref{lem.V}, \eqref{2.5op0ui8},
\eqref{2.5ploik}, $\eqref{1.1}_{1}$, $\eqref{1.1}_{3}$,
$\eqref{2.12}_{1}$, \eqref{2.4a}, \eqref{3.3ab}, the boundary
condition $\psi_{i}(0,t)=\psi_{e}(0,t)=0$, and integrating by parts,
it is direct to derive the following estimates:
\begin{equation*}\label{J1123}
\begin{aligned}[b]
|J_{1}|\leq \varphi_{i}^{2}(0,t)+\varphi_{e}^{2}(0,t)+E^{2}(0,t),
\end{aligned}
\end{equation*}
\begin{equation*}\label{J11}
\begin{aligned}[b]
J_2=&\int_{\R_{+}}\pa_x\psi_{i} \pa_x(\rho_{i}
u_{i}-\hat{\rho}\hat{u})dx+\int_{\R_{+}}\pa_x\psi_{e} \pa_x(\rho_{e}
u_{e}-\hat{\rho}\hat{u})dx-\int_{\R_{+}}(\psi_{i}+\psi_{e})\pa_x\hat{f}dx\\
=&\int_{\R_{+}}\hat{\rho}[(\pa_x\psi_{i})^{2}+(\pa_x\psi_{e})^{2}]dx+\int_{\R_{+}}\pa_x\hat{\rho}
(\psi_{i}\pa_x\psi_{i}+\psi_{e}\pa_x\psi_{e})
dx\\&+\int_{\R_{+}}[(\pa_x\psi_{i})^{2}\varphi_{i}+(\pa_x\psi_{e})^{2}\varphi_{e}]
dx+\int_{\R_{+}}(\varphi_{i}\pa_x\hat{u}\pa_x\psi_{i}+\varphi_{e}\pa_x\hat{u}\pa_x\psi_{e})
dx\\&+\int_{\R_{+}}(u_{i}\pa_x\psi_{i}\pa_x\varphi_{i}+u_{e}\pa_x\psi_{e}\pa_x\varphi_{e})
dx
-\int_{\R_{+}}(\psi_{i}+\psi_{e})\pa_x\hat{f}dx\\
\leq&\eta\|\pa_x[\varphi_{i},\varphi_{e}]\|^{2}+C_{\eta}\|\pa_x[\psi_{i},\psi_{e}]\|^{2}
+C(\eps+\tilde{\delta})\|\sqrt{\partial_{x}\hat{u}}[\varphi_{i},\varphi_{e},\psi_{i},\psi_{e}]\|^2+
C\|[\psi_{i},\psi_{e}]\|_{\infty}\|\pa_x\hat{f}\|_{L^{1}}\\
\leq&\eta\|\pa_x[\varphi_{i},\varphi_{e}]\|^{2}+C_{\eta}\|\pa_x[\psi_{i},\psi_{e}]\|^{2}
+C(\eps+\tilde{\delta})\|\sqrt{\partial_{x}\hat{u}}[\varphi_{i},\varphi_{e},\psi_{i},\psi_{e}]\|^2\\&+
C\|[\psi_{i},\psi_{e}]\|^{\frac{1}{2}}\|\partial_{x}[\psi_{i},\psi_{e}]\|^{\frac{1}{2}}
\left[\tilde{\delta}(1+t)^{-1}+\epsilon^{\theta}(1+t)^{-(1-\theta)}+\epsilon^{\frac{1}{q}}(1+t)^{-1+\frac{1}{q}}\right]\\
\leq&\eta\|\pa_x[\varphi_{i},\varphi_{e}]\|^{2}+C_{\eta}\|\pa_x[\psi_{i},\psi_{e}]\|^{2}
+C(\eps+\tilde{\delta})\|\sqrt{\partial_{x}\hat{u}}[\varphi_{i},\varphi_{e},\psi_{i},\psi_{e}]\|^2
+C(\tilde{\delta}^{\frac{4}{3}}+\epsilon^{\frac{2}{15}})(1+t)^{-\frac{7}{6}},
\end{aligned}
\end{equation*}
\begin{equation*}\label{J112}
\begin{aligned}[b]
&J_3+J_{6}+J_{7}\\=&-\frac{1}{2}\left[\frac{|u_{b}|}{\rho_{i}^{2}(0,t)}(\pa_x\varphi_{i})^{2}(0,t)+\frac{|u_{b}|}{\rho_{e}^{2}(0,t)}(\pa_x\varphi_{e})^{2}(0,t)\right]
+\int_{\R_{+}}\left(\frac{1}{2}\pa_x
u_{i}(\pa_x\varphi_{i})^{2}\rho_{i}^{-2}+\frac{1}{2}\pa_x
u_{e}(\pa_x\varphi_{e})^{2}\rho_{e}^{-2}\right)dx\\ \leq
&\int_{\R_{+}}\left(\frac{1}{2}\pa_x
\hat{u}(\pa_x\varphi_{i})^{2}\rho_{i}^{-2}+\frac{1}{2}\pa_x
\hat{u}(\pa_x\varphi_{e})^{2}\rho_{e}^{-2}\right)dx+\int_{\R_{+}}\left(\frac{1}{2}\pa_x
\psi_{i}(\pa_x\varphi_{i})^{2}\rho_{i}^{-2}+\frac{1}{2}\pa_x
\psi_{e}(\pa_x\varphi_{e})^{2}\rho_{e}^{-2}\right)dx\\
\leq&
C(\epsilon+\tilde{\delta})\|\pa_x[\varphi_{i},\varphi_{e}]\|^{2}+C\|\pa_x[\psi_{i},\psi_{e}]\|^{\frac{1}{2}}\|\pa^{2}_x[\psi_{i},\psi_{e}]\|^{\frac{1}{2}}\|\pa_x[\varphi_{i},\varphi_{e}]\|^{2}
\\ \leq &C(\epsilon+\tilde{\delta}+\varepsilon_{1})(\|\pa_x[\varphi_{i},\varphi_{e}]\|^{2})+C\varepsilon_{1}\|\pa^{2}_x[\psi_{i},\psi_{e}]\|^{2},
\end{aligned}
\end{equation*}
\begin{equation*}\label{J1123}
\begin{aligned}[b]
|J_{4}|+|J_{10}|\leq
(\eta+C\tilde{\delta}+C\epsilon)\|\pa_x[\varphi_{i},\varphi_{e}]\|^2+(C_{\eta}+C\tilde{\delta}+C\epsilon)\|\pa_x[\psi_{i},\psi_{e}]\|^2,
\end{aligned}
\end{equation*}
\begin{equation*}\label{J1mjp1345}
\begin{aligned}[b]
|J_{5}|\leq &C\int_{\R_{+}}(|\pa_x^{2}\hat{u}|+|\pa_x
\hat{\rho}|+\pa_x
\hat{u})|[\varphi_{i},\varphi_{e},\psi_{i},\psi_{e}]||\pa_x[\varphi_{i},\varphi_{e}]|dx
+C\|\pa_x[\varphi_{i},\varphi_{e}]\|\|\hat{g}\|\\
\leq&
\eta\|\pa_x[\varphi_{i},\varphi_{e}]\|^{2}+C_{\eta}(\|\hat{g}\|^{2}+\|\pa_x^{2}u^{r_{2}}\|_{\infty}^{2}+\|\pa_xu^{r_{2}}\|_{\infty}^{2})+C\tilde{\delta}\|\pa_x[\varphi_{i},\varphi_{e},\psi_{i},\psi_{e}]\|^{2}
+C\tilde{\delta}^{2}[\varphi_{i}^{2}+\varphi_{e}^{2}](0,t)\\
\leq&(\eta+C\tilde{\delta})\|\pa_x[\varphi_{i},\varphi_{e},\psi_{i},\psi_{e}]\|^{2}
+C\tilde{\delta}^{2}[\varphi_{i}^{2}(0,t)+\varphi_{e}^{2}(0,t)]+C_{\eta}\left[\epsilon^{2\theta}(1+t)^{-2(1-\frac{1}{q})}+\tilde{\delta}(1+t)^{-2}
\right]\\
\leq&(\eta+C\tilde{\delta})\|\pa_x[\varphi_{i},\varphi_{e},\psi_{i},\psi_{e}]\|^{2}
+C\tilde{\delta}^{2}[\varphi_{i}^{2}(0,t)+\varphi_{e}^{2}(0,t)]+C_{\eta}(\tilde{\delta}+\epsilon^{\frac{1}{4}})(1+t)^{-\frac{9}{5}},
\end{aligned}
\end{equation*}
\begin{equation*}\label{J112}
\begin{aligned}[b]
|J_8|\leq &
C\|\pa_x\hat{\rho}\|_{\infty}\|\pa_x[\varphi_{i},\varphi_{e}]\|\|\pa_x[\psi_{i},\psi_{e}]\|
+C\|\pa_x[\psi_{i},\psi_{e}]\|^{\frac{1}{2}}\|\pa^{2}_x[\psi_{i},\psi_{e}]\|^{\frac{1}{2}}\|\pa_x[\varphi_{i},\varphi_{e}]\|^{2}\\
\leq
&C(\epsilon+\tilde{\delta}+\varepsilon_{1})\|\pa_x[\varphi_{i},\varphi_{e},\psi_{i},\psi_{e}]\|^{2}+C\varepsilon_{1}
\|\pa^{2}_x[\psi_{i},\psi_{e}]\|^{2},
\end{aligned}
\end{equation*}
\begin{equation*}\label{J112345}
\begin{aligned}[b]
|J_{9}|+|J_{11}|\leq&C\int_{\R_{+}}(|\partial_{x}^{2}\tilde{u}|+|\partial_{x}\tilde{u}|^{2})|[\varphi_{i},\varphi_{e},\psi_{i},\psi_{e}]||\pa_x[\varphi_{i},\varphi_{e}]|dx
\\&+C\int_{\R_{+}}\left[|\partial_{x}^{2}u^{r_{2}}|+|\partial_{x}u^{r_{2}}|^{2}\right]|[\varphi_{i},\varphi_{e},\psi_{i},\psi_{e}]||\pa_x[\varphi_{i},\varphi_{e}]|dx\\
\leq&
C\tilde{\delta}^{4}[\varphi_{i}^{2}(0,t)+\varphi_{e}^{2}(0,t)]+C\tilde{\delta}\|\pa_x[\varphi_{i},\varphi_{e},\psi_{i},\psi_{e}]\|^2
\\&+C\left[\epsilon^{1+\frac{1}{q}}(1+t)^{-1+\frac{1}{q}}+\epsilon^{2\theta}(1+t)^{-2(1-\theta)}\right]\|[\psi_{i},\psi_{e}]\|\|\pa_x[\varphi_{i},\varphi_{e}]\|\\
\leq&
C(\tilde{\delta}+\epsilon^{\frac{1}{4}})\|\pa_x[\varphi_{i},\varphi_{e},\psi_{i},\psi_{e}]\|^2
+C\epsilon^{\frac{1}{4}}(1+t)^{-\frac{9}{5}},
\end{aligned}
\end{equation*}
and
\begin{equation*}\label{J1p1345}
\begin{aligned}[b]
|J_{12}|\leq &\|\pa_x[\varphi_{i},\varphi_{e}]\|\|\pa_x\hat{f}\|\leq
\eta\|\pa_x[\varphi_{i},\varphi_{e}]\|^{2}+C_{\eta}\|\pa_x\hat{f}\|^{2}\\
\leq&\eta\|\pa_x[\varphi_{i},\varphi_{e}]\|^{2}+C_{\eta}\epsilon^{1+\frac{2}{q}}(1+t)^{-2(1-\frac{1}{q})}+C_{\eta}(\tilde{\delta}+\epsilon)(1+t)^{-2}\\
\leq&\eta\|\pa_x[\varphi_{i},\varphi_{e}]\|^{2}+C_{\eta}(\tilde{\delta}+\epsilon)(1+t)^{-\frac{9}{5}},
\end{aligned}
\end{equation*}
where we take $q=10$ and $\theta=\frac{1}{8}$ in the above
estimates.

 Inserting the above estimations for $J_l$ $(1\leq l\leq 12)$ into
\eqref{sum.d1} and then choosing $\varepsilon_1$, $\epsilon$,
$\tilde{\delta}$ and $\eta$ so small, and integrating \eqref{sum.d1}
over $[0,T]$ and using \eqref{4.28de}, Cauchy-Schwarz's inequality
with $0<\eta<1$, one can see that
\begin{equation}\label{result}
\begin{aligned}[b]
&\left\|\pa_x[\varphi_{i},\varphi_{e}]\right\|^2+\int_{0}^{T}\|\sqrt{\pa_x\hat{u}}\pa_x[\varphi_{i},\varphi_{e}]\|^2
dt +\int_{0}^{T}\|\pa_x[\varphi_{i},\varphi_{e},E]\|^2dt\\
\leq&C\left(\left\|[\psi_{i0},\psi_{e0}]\right\|^2+\|E(x,0)\|^{2}+
\|[\varphi_{i0},\varphi_{e0}]\|_{H^{1}}^{2}\right)+(\eta+C\epsilon^{\frac{1}{10}}+C\tilde{\delta}^{\frac{2}{3}}+\varepsilon_{1})\|\partial_{x}^{2}[\psi_{i},\psi_{e}]\|^{2}
+C(\epsilon^{\frac{1}{10}}+\tilde{\delta}^{\frac{1}{9}}).
\end{aligned}
\end{equation}

{\bf Step 3.} {\it Dissipation of $\pa_x^{2}[\psi_{i},\psi_{e}]$.}

Multiplying $\eqref{3.2ab}_{2}$ and $\eqref{3.2ab}_{4}$ by
$-\frac{\pa^{2}_x\psi_{i}}{\rho_{i}}$ and
$-\frac{\pa^{2}_x\psi_{e}}{\rho_{e}}$ respectively, and then
integrating the resulting equations over $\R_{+}$ and taking the
summation of the resulting equations, one has
\begin{equation}\label{3.1}
\begin{aligned}[b]
&\frac{d}{dt}\int_{\R_{+}}\left(\frac{1}{2}(\pa_x\psi_{i})^{2}+\frac{1}{2}(\pa_x\psi_{e})^{2}\right)dx
+\int_{\R_{+}}\left(\frac{(\pa^{2}_x\psi_{i})^{2}}{\rho_{i}}+\frac{(\pa^{2}_x\psi_{e})^{2}}{\rho_{e}}\right)dx
\\
\qquad=&-\int_{\R_{+}}E\pa^{2}_x(\psi_{i}-\psi_{e})
dx+\int_{\R_{+}}\left(\frac{P'(\rho_{i})}{\rho_{i}}\pa_x\varphi_{i}\pa^{2}_x\psi_{i}+\frac{P'(\rho_{e})}{\rho_{e}}\pa_x\varphi_{e}\pa^{2}_x\psi_{e}\right)
dx\\
&+\int_{\R_{+}}\left(u_{i}\pa_x\psi_{i}\pa^{2}_x\psi_{i}+u_{e}\pa_x\psi_{e}\pa^{2}_x\psi_{e}\right)
dx+\int_{\R_{+}}\left(\frac{g_{i}}{\rho_{i}}\pa^{2}_x\psi_{i}+\frac{g_{e}}{\rho_{e}}\pa^{2}_x\psi_{e}\right)dx\\=&\sum\limits_{l=13}^{16}J_{l},
\end{aligned}
\end{equation}
where we have used the boundary condition
$\psi_{i}(0,t)=\psi_{e}(0,t)=0.$

 We now turn to estimate $J_l$ $(13\leq J\leq
16)$ term by term. By applying  Cauchy-Schwarz's inequality with
$0<\eta<1$, Sobolev inequality \eqref{sob.}, Lemma \ref{cl.Re.Re2.},
Lemma \ref{lem.V},  \eqref{2.4a} and integrating by parts,
 we can obtain  that
\begin{equation*}\label{3.p}
\begin{split}
J_{13}=&E(0,t)[(\pa_x\psi_{i})(0,t)-(\pa_x\psi_{e})(0,t)]+\int_{\R_{+}}\pa_xE\pa_x(\psi_{i}-\psi_{e})dx\\
\leq&
\eta[(\pa_x\psi_{i})^{2}(0,t)+(\pa_x\psi_{e})^{2}(0,t)]+C_{\eta}E^{2}(0,t)
+\frac{1}{2}\|\pa_x[\psi_{i},\psi_{e}]\|^{2}+\frac{1}{2}\|\pa_xE\|^{2}\\
\leq&\eta\|\pa_x[\psi_{i},\psi_{e}]\|_{\infty}^{2}
+C_{\eta}E^{2}(0,t)
+\frac{1}{2}\|\pa_x[\psi_{i},\psi_{e}]\|^{2}+\frac{1}{2}\|\pa_xE\|^{2}\\
\leq&\eta(\|\pa_x[\psi_{i},\psi_{e}]\|^{2}+\|\pa^{2}_x[\psi_{i},\psi_{e}]\|^{2})
+C_{\eta}E^{2}(0,t)
+\frac{1}{2}\|\pa_x[\psi_{i},\psi_{e}]\|^{2}+\frac{1}{2}\|\pa_xE\|^{2},
\end{split}
\end{equation*}
\begin{equation*}\label{3.2asd}
\begin{split}
|J_{14}|+|J_{15}|\leq\eta\|\pa^{2}_x[\psi_{i},\psi_{e}]\|^{2}+C_{\eta}\|\pa_x[\varphi_{i},\varphi_{e},\psi_{i},\psi_{e}]\|^{2}
\end{split}
\end{equation*}
and
\begin{equation*}\label{J1mjp1345}
\begin{aligned}
|J_{16}|\leq
(\eta+C\tilde{\delta})\|\pa_x[\varphi_{i},\varphi_{e},\psi_{i},\psi_{e},
\pa_x\psi_{i},\pa_x\psi_{e}]\|^{2}
+C\tilde{\delta}^{2}[\varphi_{i}^{2}(0,t)+\varphi_{e}^{2}(0,t)]+C_{\eta}(\tilde{\delta}+\epsilon^{\frac{1}{4}})(1+t)^{-\frac{9}{5}},
\end{aligned}
\end{equation*}
where we take $q=10$  in the above estimates  and the estimate of
$J_{16}$ is the same as $J_{5}.$

Inserting  the above estimations for $J_l$ $(13\leq J\leq 16)$ into
\eqref{3.1} and then integrating \eqref{3.1} over $[0,T]$ and using
\eqref{4.28de} and \eqref{result}, one can see that
\begin{equation}\label{3.7}
\begin{aligned}[b]
&\left\|\pa_x[\psi_{i},\psi_{e}]\right\|^2+\int_{0}^{T}\|\pa^{2}_x[\psi_{i},\psi_{e}]\|^2dt\\
\leq&
C\left(\|[\varphi_{i0},\varphi_{e0},\psi_{i0},\psi_{e0}]\|_{H^{1}}^{2}+\left\|[E(x,0)]\right\|^2\right)+C\left(\epsilon^{\frac{1}{10}}+\tilde{\delta}^{\frac{1}{9}}\right).
\end{aligned}
\end{equation}
where we choose $\varepsilon_1$, $\epsilon$, $\tilde{\delta}$ and
$\eta$ sufficiently small.
\begin{proof}[Proof of Proposition \ref{priori.est}]

Now, following Step 1, Step 2 and Step 3, we are ready to prove
Proposition \ref{priori.est}. Summing up the estimates
\eqref{4.28de}, \eqref{result}, $\eqref{3.7}$ and taking $\epsilon$,
$\tilde{\delta}$, $\varepsilon_{1}$, $\eta$  suitably small, we have
\begin{equation}\label{resultuijklo}
\begin{aligned}[b]
&\sup_{0\leq t\leq
T}\left(\left\|[\varphi_{i},\varphi_{e},\psi_{i},\psi_{e}]\right\|_{H^1}^2+\|E\|^{2}\right)
+\int_{0}^{T}\|\sqrt{\pa_x\hat{u}}[\varphi_{i},\varphi_{e},\psi_{i},\psi_{e}]\|^2
dt
\\&+\int_{0}^{T}\|\pa_x[\varphi_{i},\varphi_{e},E]\|^2dt
+\int_{0}^{T}\|\pa_x\left[\psi_{i},\psi_{e}\right]\|_{H^1}^2dt\\
\leq&
C\left(\|[\varphi_{i0},\varphi_{e0},\psi_{i0},\psi_{e0}]\|_{H^{1}}^{2}+\left\|[E(x,0)]\right\|^2+\epsilon^{\frac{1}{10}}+\tilde{\delta}^{\frac{1}{9}}\right).
\end{aligned}
\end{equation}
From $\eqref{3.2ab}_{5}$, it follows
$$
\|\pa_xE\|^2\leq \|[\varphi_i,\varphi_e]\|^2,
$$
this and \eqref{resultuijklo} imply the desired estimate
\eqref{resultuijk}. Thus the proof of Proposition \ref{priori.est}
is completed.
\end{proof}

\section{Global existence and large time behavior}
We are now in a position to complete the proof of Theorem
\ref{main.Res.}.

\begin{proof}[Proof of Theorem \ref{main.Res.}]
By the {\it a priori} estimates \eqref{resultuijk}, there exists a
positive constant $C_{0}$ such that
\begin{equation}\label{rlpo0f}
\begin{aligned}[b]
\left\|[\varphi_{i},\varphi_{e},\psi_{i},\psi_{e},E]\right\|_{H^1}^2\leq
C_{0}\left(\|[\varphi_{i0},\varphi_{e0},\psi_{i0},\psi_{e0}]\|_{H^{1}}^{2}
+\left\|[E(x,0)]\right\|^2+\epsilon^{\frac{1}{10}}+\tilde{\delta}^{\frac{1}{9}}\right)
\end{aligned}
\end{equation}
holds. It is straightforward to see that there exists a small
constant $\varepsilon_{0}$ such that if
$$\|[\varphi_{i0},\varphi_{e0},\psi_{i0},\psi_{e0}]\|_{H^{1}}^{2}+\left\|E(0,x)\right\|^2
\leq\varepsilon_{0}^{2},$$ we can close the {\it a priori}
assumption \eqref{2.5op0ui8} by choosing
$\varepsilon_{1}=4\sqrt{C_{0}(\varepsilon_{0}^{2}+\epsilon^{\frac{1}{10}}+\tilde{\delta}^{\frac{1}{9}})}$.
By letting $\epsilon$ and $\tilde{\delta}$ be small enough, then the
global existence of the solution of \eqref{3.2ab} follows from the
standard continuation argument based on the local existence and the
{\it a priori} estimates in Proposition \ref{priori.est}. Moreover,
\eqref{rlpo0f} and \eqref{retghy} imply \eqref{resultqa21hl}. Our
intention next is to prove the large time behavior as
\eqref{resultqa216y89} and \eqref{resultqa21gt}. For this, we first
justify the following limits:
\begin{equation}\label{latm1}
\lim\limits_{t\rightarrow+\infty}\left\|\pa_x[\varphi_{i},\varphi_{e},\psi_{i},\psi_{e}](t)
\right\|_{L^2}^2= 0,
\end{equation}
and
\begin{equation}\label{latm2}
\lim\limits_{t\rightarrow+\infty}\left\|\pa_xE(t)\right\|^2 = 0.
\end{equation}
To prove \eqref{latm1} and \eqref{latm2}, we get from \eqref{d.rho},
\eqref{d.rhode}, \eqref{3.1} and \eqref{resultuijk} that
\begin{equation}\label{latm3}
\begin{aligned}[b]
&\int_{0}^{+\infty}\left|\frac{d}{dt}\left\|\pa_x[\varphi_{i},\varphi_{e},\psi_{i},\psi_{e}]\right\|^2\right|dt
\\=&2\int_{0}^{+\infty}\left[\left|\int_{\R_{+}}\pa_t\pa_x\varphi_{i}\pa_x\varphi_{i}dx\right|+\left|\int_{\R_{+}}\pa_t\pa_x\varphi_{e}\pa_x\varphi_{e}dx\right|\right]dt
+\int_{0}^{+\infty}\left|\frac{d}{dt}\left\|\pa_x[\psi_{i},\psi_{e}]\right\|^2\right|dt
\\ \leq& C+C\int_{0}^{+\infty}
\left\|\pa_x\left[\varphi_{i},\varphi_{e},\psi_{i},\psi_{e},E,\pa_x\left[\psi_{i},\psi_{e}\right]\right]\right\|^2dt<+\infty.
\end{aligned}
\end{equation}
On the other hand, $\eqref{3.2ab}_{5}$, $\eqref{3.2ab}_{1}$,
$\eqref{3.2ab}_{3}$ and \eqref{resultuijk} yield
\begin{equation}\label{latm4}
\begin{aligned}[b]
&\int_{0}^{+\infty}\left|\frac{d}{dt}\left\|\pa_xE\right\|^2\right|dt
=2\int_{0}^{+\infty}\left|\int_{\R_{+}}\pa_t\pa_xE\pa_xE
dx\right|dt\\=&2\int_{0}^{+\infty}\left|\int_{\R_{+}}\left(\pa_t\varphi_{i}-\pa_t\varphi_{e}\right)\pa_xE
dx\right|dt <+\infty.
\end{aligned}
\end{equation}
Consequently, \eqref{latm3}, \eqref{latm4} together with
\eqref{resultuijk} gives \eqref{latm1} and \eqref{latm2}. Then
\eqref{resultqa216y89} and \eqref{resultqa21gt} follows from
\eqref{latm1}, \eqref{latm2} and Sobolev's inequality \eqref{sob.}.
This ends the proof of Theorem \ref{main.Res.}.
\end{proof}

\section{Appendix}
In this appendix, we will give the following inequalities stated in
 Lemma \ref{lem.V} repeatedly used in the paper.
\begin{lemma}\label{lem.V}
(i) For any function h and $(k+1)j>2,$ there is a positive constant
C such that,
\begin{eqnarray}\label{4.1}
\int_{\mathbb{R}_{+}}|\partial_{x}^{k}(\tilde{u}-u_{*})|^{j}|h|^{2}dx\leq
C\tilde{\delta}^{(k+1)j-2}\left[\tilde{\delta}h^{2}(0,t)+\|\partial_{x}h(t)\|^{2}\right].
\end{eqnarray}

(ii) For any functions f, h and $2(k+1)j>3,$ there is a positive
constant C such that,
\begin{eqnarray}\label{4.2}
\begin{aligned}
\int_{\mathbb{R}_{+}}|\partial_{x}^{k}(\tilde{u}-u_{*})|^{j}|h\partial_{x}f|dx
\leq
\tilde{\delta}\|\partial_{x}f(t)\|^{2}+C\tilde{\delta}^{2(k+1)j-3}\left[\tilde{\delta}h^{2}(0,t)+\|\partial_{x}h(t)\|^{2}\right].
\end{aligned}
\end{eqnarray}

(iii) For any $\theta\in [0,1]$, we have
\begin{eqnarray}\label{4.2a}
\|\partial_{x}(n^{r_{2}}-n_{*}),\partial_{x}(u^{r_{2}}-u_{*})
 \|_{\infty}\leq
C\epsilon^{\theta}(1+t)^{-(1-\theta)}.
\end{eqnarray}

(iv) For any $\theta\in [0,1]$, $q\geq 10$, we have
\begin{eqnarray}\label{4.2ab}
\begin{aligned}
\int_{\mathbb{R}_{+}}\left(|\hat{f}|+|\hat{g}+\partial_{x}^{2}u^{r_{2}}|\right)dx
\leq
C\frac{\tilde{\delta}}{1+\tilde{\delta}t}+C\epsilon^{\theta}(1+t)^{-(1-\theta)}\ln(1+\tilde{\delta}t)
\end{aligned}
\end{eqnarray}
and \begin{eqnarray}\label{4.2abc}
\begin{aligned}
\int_{\mathbb{R}_{+}}|\partial_{x}\hat{f}|dx \leq
C\tilde{\delta}(1+t)^{-1}+C\epsilon^{\theta}(1+t)^{-(1-\theta)}+C\epsilon^{\frac{1}{q}}(1+t)^{-1+\frac{1}{q}}.
\end{aligned}
\end{eqnarray}

(v) For $q\geq 10$, we have
\begin{eqnarray}\label{4.2abcdv}
\begin{aligned}
\int_{\mathbb{R}_{+}}|\hat{g}|^{2}dx \leq
C\epsilon^{1+\frac{2}{q}}(1+t)^{-2(1-\frac{1}{q})}+C\tilde{\delta}(1+t)^{-2}
\end{aligned}
\end{eqnarray}
and \begin{eqnarray}\label{4.2abcf}
\begin{aligned}
\int_{\mathbb{R}_{+}}|\partial_{x}\hat{f}|^{2}dx \leq
C\epsilon^{1+\frac{2}{q}}(1+t)^{-2(1-\frac{1}{q})}+C(\tilde{\delta}+\epsilon)(1+t)^{-2}.
\end{aligned}
\end{eqnarray}
\end{lemma}

\begin{proof}
$(i)$ Using \eqref{2.4a} and the following Poincar\'{e} type
inequalities
\begin{eqnarray}\label{4.2ab1q}
\begin{aligned}
|h(x,t)|\leq |h(0,t)|+x^{\frac{1}{2}}\|\pa_xh(t)\|,
\end{aligned}
\end{eqnarray} for $(k+1)j>2,$ we
have
\begin{eqnarray*}\label{4.1}
\begin{split}
&\int_{\mathbb{R}_{+}}|\partial_{x}^{k}(\tilde{u}-u_{*})|^{j}|h|^{2}dx\\
\leq&\int_{\mathbb{R}_{+}}|\partial_{x}^{k}(\tilde{u}-u_{*})|^{j}\left(h^{2}(0,t)+x\|\partial_{x}h(t)\|^{2}\right)dx\\
\leq& Ch^{2}(0,t)\int_{\mathbb{R}_{+}}\frac{\tilde{\delta}^{(k+1)j}}
{(1+\tilde{\delta}x)^{(k+1)j}}dx+C\|\partial_{x}h(t)\|^{2}\int_{\mathbb{R}_{+}}\frac{x\tilde{\delta}^{(k+1)j}}
{(1+\tilde{\delta}x)^{(k+1)j}}dx \\
\leq&
C\tilde{\delta}^{(k+1)j-2}\left[\tilde{\delta}h^{2}(0,t)+\|\partial_{x}h(t)\|^{2}\right].
\end{split}
\end{eqnarray*}

$(ii)$ By the Young inequality and Lemma \ref{lem.V} $(i)$, for
$2(k+1)j>3,$ we have
\begin{eqnarray*}\label{4.2}
\begin{aligned}
&\int_{\mathbb{R}_{+}}|\partial_{x}^{k}(\tilde{u}-u_{*})|^{j}|h\partial_{x}f|dx\\
 \leq&
\tilde{\delta}\|\partial_{x}f(t)\|^{2}+C\int_{\mathbb{R}_{+}}\frac{\tilde{\delta}^{2(k+1)j-1}}
{(1+\tilde{\delta}x)^{2(k+1)j}}h^{2}dx\\
\leq&
\tilde{\delta}\|\partial_{x}f(t)\|^{2}+C\tilde{\delta}^{2(k+1)j-3}\left[\tilde{\delta}h^{2}(0,t)+\|\partial_{x}h(t)\|^{2}\right].
\end{aligned}
\end{eqnarray*}

$(iii)$ From Lemma \ref{cl.Re.Re2.} $(ii)$, we have
\begin{eqnarray*}\label{4.2a}
\begin{aligned}
\|\partial_{x}(n^{r_{2}}-n_{*}),\partial_{x}(u^{r_{2}}-u_{*})
 \|_{\infty}
\leq C\min\{\epsilon,(1+t)^{-1}\}.
\end{aligned}
\end{eqnarray*}
Thus we have
\begin{eqnarray*}\label{4.2a}
\begin{aligned}
\|\partial_{x}(n^{r_{2}}-n_{*}),\partial_{x}(u^{r_{2}}-u_{*})\|_{\infty}
\leq C\epsilon^{\theta}(1+t)^{-(1-\theta)}.
\end{aligned}
\end{eqnarray*}
Here we have used the fact that if $0<C\leq A$ and  $0<C\leq B$,
then $C\leq A^{\theta}B^{1-\theta}$ for any $0\leq \theta\leq 1.$

$(iv)$ Using \eqref{2.14}, \eqref{2.4a}, Lemma \ref{cl.Re.Re2.}
$(ii)$ and Lemma \ref{lem.V} $(iii)$, we have
\begin{eqnarray*}\label{4.2ab}
\begin{aligned}
&\int_{\mathbb{R}_{+}}\left(|\hat{f}|+|\hat{g}+\partial_{x}^{2}u^{r_{2}}|\right)dx\\
\leq &
C\int_{\mathbb{R}_{+}}\{\partial_{x}\tilde{u}(u^{r_{2}}-u_{*})+\partial_{x}u^{r_{2}}(u_{*}-\tilde{u})\}dx\\
=&C\int_{\mathbb{R}_{+}}\partial_{x}[(u^{r_{2}}-u_{*})(\tilde{u}-u_{*})]dx+2C\int_{\mathbb{R}_{+}}\partial_{x}u^{r_{2}}(u_{*}-\tilde{u})dx\\
=&2C\int_{0}^{t}\partial_{x}u^{r_{2}}(u_{*}-\tilde{u})dx+2C\int_{t}^{+\infty}\partial_{x}u^{r_{2}}(u_{*}-\tilde{u})dx\\
\leq&
C\|\partial_{x}u^{r_{2}}\|_{\infty}\int_{0}^{t}\frac{\tilde{\delta}}{1+\tilde{\delta}x}dx+C\frac{\tilde{\delta}}{1+\tilde{\delta}t}
\int_{t}^{+\infty}\partial_{x}u^{r_{2}}dx\\
\leq&
C\|\partial_{x}u^{r_{2}}\|_{\infty}\ln(1+\tilde{\delta}t)+C\frac{\tilde{\delta}}{1+\tilde{\delta}t}
\|\partial_{x}u^{r_{2}}\|_{L^{1}}\\ \leq&
C\epsilon^{\theta}(1+t)^{-(1-\theta)}\ln(1+\tilde{\delta}t)+C\frac{\tilde{\delta}}{1+\tilde{\delta}t}.
\end{aligned}
\end{eqnarray*}
where we have used $u^{r_{2}}(0,t)=u_{*}$ and $\tilde{u}\rightarrow
u_{*}$ as $x\rightarrow +\infty.$

 Similarly,  we can
obtain that
\begin{eqnarray*}\label{4.2abc}
\begin{aligned}
\int_{\mathbb{R}_{+}}|\partial_{x}\hat{f}|dx\leq&
C\int_{\mathbb{R}_{+}}
\left\{(|\partial_{x}^{2}\tilde{u}|+(\partial_{x}\tilde{u})^{2})(u^{r_{2}}-u_{*})+\partial_{x}\tilde{u}\partial_{x}u^{r_{2}}+|\partial_{x}^{2}u^{r_{2}}|+(\partial_{x}u^{r_{2}})^{2}\right\}dx\\
\leq&C\|\partial_{x}u^{r_{2}}\|_{\infty}\int_{\mathbb{R}_{+}}x(|\partial_{x}^{2}\tilde{u}|+(\partial_{x}\tilde{u})^{2})dx
+C\|\partial_{x}u^{r_{2}}\|_{\infty}\|\partial_{x}\tilde{u}\|_{L^{1}}+C\|\partial_{x}^{2}u^{r_{2}}\|_{L^{1}}+C\|\partial_{x}u^{r_{2}}\|^{2}\\
\leq&
C\tilde{\delta}(1+t)^{-1}+C\epsilon^{\theta}(1+t)^{-(1-\theta)}+C\epsilon^{\frac{1}{q}}(1+t)^{-1+\frac{1}{q}},
\end{aligned}
\end{eqnarray*}
where we have used the fact that $u^{r_{2}}(0,t)=u_{*}$ which yields
$u^{r_{2}}(x,t)-u_{*}\leq x\|\partial_{x}u^{r_{2}}\|_{\infty}.$

$(v)$ Noticing \eqref{2.14} and the fact that
$u^{r_{2}}(x,t)-u_{*}\leq x\|\partial_{x}u^{r_{2}}\|_{\infty}$, and
applying Lemma \ref{cl.Re.Re2.} and \eqref{2.4a}, we obtain that
\begin{eqnarray*}\label{4.2abcdv}
\begin{aligned}
\int_{\mathbb{R}_{+}}|\hat{g}|^{2}dx \leq&
C\int_{\mathbb{R}_{+}}\{|\partial_{x}^{2}u^{r_{2}}|^{2}+|\partial_{x}\tilde{u}|^{2}|(u^{r_{2}}-u_{*})|^{2}+|\partial_{x}u^{r_{2}}|^{2}|(u_{*}-\tilde{u})|^{2}\}dx\\
\leq&
C\|\partial_{x}^{2}u^{r_{2}}\|^{2}+C\|\partial_{x}u^{r_{2}}\|_{\infty}^{2}\int_{\mathbb{R}_{+}}|\partial_{x}\tilde{u}|^{2}x^{2}dx
+C\|\partial_{x}u^{r_{2}}\|_{\infty}^{2}\int_{\mathbb{R}_{+}}(u_{*}-\tilde{u})|^{2}dx\\
\leq&
C\epsilon^{1+\frac{2}{q}}(1+t)^{-2(1-\frac{1}{q})}+C\tilde{\delta}(1+t)^{-2}
\end{aligned}
\end{eqnarray*}
and \begin{eqnarray*}\label{4.2abcf}
\begin{aligned}
\int_{\mathbb{R}_{+}}|\hat{f}_{x}|^{2}dx\leq& C\int_{\mathbb{R}_{+}}
\left\{(|\partial_{x}^{2}\tilde{u}|^{2}+(\partial_{x}\tilde{u})^{4})|(u^{r_{2}}-u_{*})|^{2}+|\partial_{x}\tilde{u}|^{2}(\partial_{x}u^{r_{2}})^{2}
+|\partial_{x}^{2}u^{r_{2}}|^{2}+(\partial_{x}u^{r_{2}})^{4}\right\}dx\\
\leq&
C\|\partial_{x}^{2}u^{r_{2}}\|^{2}+C\|\partial_{x}u^{r_{2}}\|_{\infty}^{3}\|\partial_{x}u^{r_{2}}\|_{L^{1}}+C\|\partial_{x}u^{r_{2}}\|_{\infty}^{2}\int_{\mathbb{R}_{+}}
\left[(|\partial_{x}^{2}\tilde{u}|^{2}+(\partial_{x}\tilde{u})^{4})x^{2}+|\partial_{x}\tilde{u}|^{2}\right]dx
\\
\leq&
C\epsilon^{1+\frac{2}{q}}(1+t)^{-2(1-\frac{1}{q})}+C(\tilde{\delta}+\epsilon)(1+t)^{-2}.
\end{aligned}
\end{eqnarray*}
\end{proof}

\medskip
\noindent {\bf Acknowledgements:} The research was supported by the
National Natural Science Foundation of China \#11331005, the Program
for Changjiang Scholars and Innovative Research Team in University
\#IRT13066, and the Scientific Research Funds of Huaqiao University
(Grant No.15BS201). The first author would like to thank Professor
Renjun Duan for many fruitful discussions on the topic of the paper.

\bigbreak

\end{document}